\documentclass[a4paper,12pt]{amsart}

\usepackage[a4paper, top=3cm, left=2.5cm, right=2.5cm, bottom=3cm]{geometry}

\usepackage{hyperref}
\usepackage{amsmath}
\usepackage{amsthm, pb-diagram}
\usepackage{amssymb,comment}
\usepackage{amsfonts,graphicx,color}
\usepackage{enumerate, fancyhdr, dsfont}
\usepackage[normalem]{ulem}

    \newcommand{\dom}{\mbox{\rm dom}}

    \newcommand{\ran}{\mbox{\rm ran}}

    \newcommand{\thzfc}{\mathrm{ZFC}}

    \newcommand{\Bwf}{\mathcal{B}}

    \newcommand{\Iwf}{\mathcal{I}}
    
    \newcommand{\Mwf}{\mathcal{M}}
    \newcommand{\Nwf}{\mathcal{N}}
    
    \newcommand{\Pwf}{\mathcal{P}}
    \newcommand{\Swf}{\mathcal{S}}

    \newcommand{\afrak}{\mathfrak{a}}
    \newcommand{\bfrak}{\mathfrak{b}}
    \newcommand{\cfrak}{\mathfrak{c}}
    \newcommand{\dfrak}{\mathfrak{d}}

    \newcommand{\pfrak}{\mathfrak{p}}

    \newcommand{\menos}{\smallsetminus}

    \newcommand{\frestr}{\!\!\upharpoonright\!\!}

    \newcommand{\add}{\mbox{\rm add}}
    \newcommand{\cov}{\mbox{\rm cov}}
    \newcommand{\non}{\mbox{\rm non}}
    \newcommand{\cof}{\mbox{\rm cof}}
    \newcommand{\limdir}{\mbox{\rm limdir}}

    \newcommand{\Bor}{\mathds{B}}
    \newcommand{\Cor}{\mathds{C}}
    \newcommand{\Dor}{\mathds{D}}
    \newcommand{\Eor}{\mathds{E}}
    \newcommand{\Hor}{\mathds{H}}
    
    \newcommand{\Loc}{\mathds{LOC}}
    
    \newcommand{\Por}{\mathds{P}}
    \newcommand{\Pbb}{\mathds{P}}
    \newcommand{\Qor}{\mathds{Q}}
    
    \newcommand{\Sor}{\mathds{S}}
    
    \newcommand{\Qnm}{\dot{\mathds{Q}}}

    \newcommand{\Ncal}{\mathcal{N}}
    \newcommand{\Mcal}{\mathcal{M}}
    \newcommand{\Ical}{\mathcal{I}}
    \newcommand{\SNcal}{\mathcal{SN}}

    \newcommand{\cf}{\mbox{\rm cf}}

    \newcommand{\la}{\langle}
    \newcommand{\ra}{\rangle}

   \newcommand{\Abf}{\mathbf{A}}

\newcommand{\Dbf}{\mathbf{D}}
\newcommand{\mbf}{\mathbf{m}}

\newcommand{\sbf}{\mathbf{s}}

\newcommand{\seq}{\mathrm{seq}}
\newcommand{\Fr}{\mathrm{Fr}}
\newcommand{\Rbf}{\mathbf{R}}
\newcommand{\D}{\mathbf{D}}
\newcommand{\Cn}{\mathbf{Cn}}
\newcommand{\Lc}{\mathbf{Lc}}

\newcommand{\Hcal}{\mathcal{H}}
\newcommand{\Scal}{\mathcal{S}}
\newcommand{\id}{\mathrm{id}}
\newcommand{\Ed}{\mathbf{Ed}}

\newcommand{\hgt}{\mathrm{ht}}
\newcommand{\Id}{\mathbf{Id}}
\newcommand{\Md}{\mathbf{Md}}

\newcommand{\rojo}[1]{{\color{red}#1}}

\newcommand{\Diego}[1]{{\color{blue}Diego: #1}}

\title[Filter-linkedness]{Filter-linkedness and its effect on preservation of cardinal characteristics}

\author{J\"org Brendle}
\address{Graduate School of System Informatics, Kobe University,
Rokko--dai 1--1, Nada--ku, 657--8501 Kobe, Japan}
\email{brendle@kobe-u.ac.jp}

\author{Miguel A. Cardona}
\address{TU Wien, Faculty of Mathematics and Geoinformation, Institute of Discrete Mathematics and Geometry, Wiedner Hauptstrasse 8--10, A--1040 Vienna, Austria }
\email{miguel.montoya@tuwien.ac.at}
\urladdr{https://www.researchgate.net/profile/Miguel\_Cardona\_Montoya}

\author{Diego A. Mej\'ia}
\address{Creative Science Course (Mathematics), Faculty of Science, Shizuoka University, Ohya 836, Suruga--ku, 422--8529 Shizuoka, Japan}
\email{diego.mejia@shizuoka.ac.jp}
\urladdr{http://www.researchgate.com/profile/Diego\_Mejia2}

\thanks{The first author was partially supported by Grants-in-Aid for Scientific Research (C) 15K04977 and 18K03398, Japan Society for the Promotion of Science; the second author was supported by the Austrian Science Fund (FWF) P30666 and he is a recipient of a DOC Fellowship of the Austrian Academy of Sciences at the Institute of Discrete Mathematics and Geometry, TU Wien; and the third author was supported by the grant no. IN201711, Direcci\'on Operativa de Investigaci\'on, Instituci\'on Universitaria Pascual Bravo, and by Grant-in-Aid for Early Career Scientists 18K13448, Japan Society for the Promotion of Science.}



\begin{document}

\makeatletter
\def\@roman#1{\romannumeral #1}
\makeatother

\newcounter{enuAlph}
\renewcommand{\theenuAlph}{\Alph{enuAlph}}

\renewcommand{\theequation}{\thesection.\arabic{equation}}

\theoremstyle{plain}
  \newtheorem{theorem}{Theorem}[section]
  \newtheorem{corollary}[theorem]{Corollary}
  \newtheorem{lemma}[theorem]{Lemma}
  \newtheorem{mainlemma}[theorem]{Main Lemma}
  \newtheorem{prop}[theorem]{Proposition}
  \newtheorem{clm}[theorem]{Claim}
  \newtheorem{exer}[theorem]{Exercise}
  \newtheorem{question}[theorem]{Question}
  \newtheorem{problem}[theorem]{Problem}
  \newtheorem{conjecture}[theorem]{Conjecture}
  \newtheorem*{thm}{Theorem}
  \newtheorem{teorema}[enuAlph]{Theorem}
  \newtheorem*{corolario}{Corollary}
\theoremstyle{definition}
  \newtheorem{definition}[theorem]{Definition}
  \newtheorem{example}[theorem]{Example}
  \newtheorem{remark}[theorem]{Remark}
  \newtheorem{notation}[theorem]{Notation}
  \newtheorem{context}[theorem]{Context}

  \newtheorem*{defi}{Definition}
  \newtheorem*{acknowledgements}{Acknowledgements}



\begin{abstract}
   We introduce the property ``$F$-linked'' of subsets of posets for a given free filter $F$ on the natural numbers, and define the properties ``$\mu$-$F$-linked'' and ``$\theta$-$F$-Knaster'' for posets in a natural way. We show that $\theta$-$F$-Knaster posets preserve strong types of unbounded families and of maximal almost disjoint families.

   Concerning iterations of such posets, we develop a general technique to construct $\theta$-$\Fr$-Knaster posets (where $\Fr$ is the Frechet ideal) via matrix iterations of ${<}\theta$-ultrafilter-linked posets (restricted to some level of the matrix). This is applied to prove consistency results about Cicho\'n's diagram (without using large cardinals) and to prove the consistency of the fact that, for each Yorioka ideal, the four cardinal invariants associated with it are pairwise different.

   At the end, we show that three strongly compact cardinals are enough to force that Cicho\'n's diagram can be separated into $10$ different values.
\end{abstract}

\maketitle

\section{Introduction}\label{SecIntro}

The third author~\cite{mejiavert} introduced the notion of \emph{Frechet-linkedness} (abbreviated \emph{$\Fr$-linkedness}) inspired by Miller's proof that $\Eor$, the standard $\sigma$-centered poset that adds an eventually different real (see Definition~\ref{DefEDforcing}), does not add dominating reals. The third author showed that $\Eor$ and random forcing are $\sigma$-$\Fr$-linked, and that no $\sigma$-$\Fr$-linked poset adds dominating reals. Moreover, he showed that such posets preserve a certain type of mad (maximal almost disjoint) families (like those added by Hechler's poset $\Hor_\kappa$ for adding a mad family of size $\kappa$).

Frechet-linkedness is a notion of subsets of posets: given a poset $\Por$ and $Q\subseteq\Por$, \emph{$Q$ is Frechet-linked (in $\Por$)} if, for any countable sequence $\la p_n:n<\omega\ra$ of members of $Q$, there is some $q\in\Por$ forcing that $p_n$ is in the generic filter for infinitely many $n$. Given a cardinal $\mu$, a poset is \emph{$\mu$-$\Fr$-linked} if it is the union of $\mu$-many $\Fr$-linked subsets. A Knaster-type notion can be defined in the natural way: a poset $\Por$ is \emph{$\theta$-$\Fr$-Knaster} if any subset of $\Por$ of size $\theta$ contains a $\Fr$-linked subset of the same size. It is clear that any $\mu$-$\Fr$-linked poset is $\mu^+$-$\Fr$-Knaster and, for regular $\theta$, any $\theta$-$\Fr$-Knaster poset satisfies the $\theta$-Knaster property (see Remark~\ref{RemKnaster}).

The notion $\Fr$-Knaster appears implicitly in several places. For example, using finitely additive measures along FS (finite support) iterations, Shelah~\cite{ShCov} constructed an $\aleph_1$-$\Fr$-Knaster poset to force $\cov(\Nwf)$ with countable cofinality, while Kellner, Shelah and T\v{a}nasie~\cite{KST} used the same technique to construct a $\theta$-$\Fr$-Knaster poset that forces
\begin{equation}\label{leftKST}
    \aleph_1<\add(\Nwf)<\bfrak=\theta<\cov(\Nwf)<\non(\Mwf)<\cov(\Mwf)=\cfrak.
\end{equation}
Using the analog of this technique for ultrafilters, Goldstern, Shelah and the third author~\cite{GMS} constructed a $\theta$-$\Fr$-Knaster poset that forces
\begin{equation}\label{leftGMS}
    \aleph_1<\add(\Nwf)<\cov(\Nwf)<\bfrak=\theta<\non(\Mwf)<\cov(\Mwf)=\cfrak.
\end{equation}
These two results state the both possible ways to separate the cardinal invariants in the left side of Cicho\'n's diagram (Figure~\ref{FigCichon}). We assume that the reader is familiar with this diagram.

\begin{figure}
\begin{center}
  \includegraphics{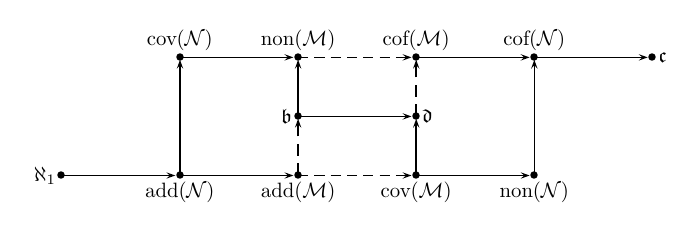}
  \caption{Cicho\'n's diagram. The arrows mean $\leq$ and dotted arrows represent
  $\add(\Mwf)=\min\{\bfrak,\cov(\Mwf)\}$ and $\cof(\Mwf)=\max\{\dfrak,\non(\Mwf)\}$.}
  \label{FigCichon}
\end{center}
\end{figure}

The main challenge in both results is to force $\bfrak=\theta$ while iterating restrictions of $\Eor$ to small models (for (\ref{leftGMS})), or similar restrictions of random forcing and of a variation of $\Eor$ (for (\ref{leftKST})). In fact, from both arguments, it can be inferred that $\theta$-$\Fr$-Knaster posets preserve a strong type of unbounded families (see Theorem~\ref{FrKnasterpresunb}).

Hechler's poset $\Hor_\theta$ for adding a mad family of size $\theta$ was originally defined in~\cite{BrF} (see Definition~\ref{DefHz}).
Since $\mu$-$\Fr$-linked posets preserve the mad family added by $\Hor_\theta$ for $\mu<\theta$, it is natural to ask:

\begin{question}\label{Qmad}
   If $\theta$ is a regular uncountable cardinal, does any $\theta$-$\Fr$-Knaster poset preserve the mad family added by $\Hor_\theta$?
\end{question}

Such a mad family falls into the category of what we call \emph{$\theta$-strong-$\Md$ a.d.\ family} (see Definition~\ref{DefMd} and Lemma~\ref{HechlerMad}), which are also preserved by $\mu$-$\Fr$-linked posets for $\mu<\theta$ according to~\cite{mejiavert}. Moreover, it was proved in~\cite{FFMM,mejiavert} that a large class of FS iterations preserve the mad family added by $\Hor_\theta$ for $\theta$ regular, which is used to prove that $\afrak=\bfrak$ can be forced (where $\afrak$ is the minimal size of an infinite mad family) in various models where Cicho\'n's diagram is divided into several values. In fact, this class is contained in the class of FS iterations of $\mu$-$\Fr$-linked posets with $\mu<\theta$, but since any such iteration yields a $\theta$-$\Fr$-Knaster poset (see~\cite[Sect.~5]{mejiavert} or Lemma~\ref{Flinkit}), the previous argument is nicely generalized with a positive answer to Question~\ref{Qmad}. Even more, such a positive answer will imply that it can be forced, in addition, that $\afrak=\theta$ in both (\ref{leftKST}) and (\ref{leftGMS}).

In this paper, we answer Question~\ref{Qmad} in the positive.

\begin{teorema}[Theorem~\ref{FrKnastermad}]\label{Kmad}
    If $\theta$ is a regular uncountable cardinal then any $\theta$-$\Fr$-Knaster poset preserves all the $\theta$-strong-$\Md$ (a.d.) families from the ground model.
\end{teorema}

\begin{corolario}
    In both (\ref{leftKST}) and (\ref{leftGMS}) it can be forced, in addition, that $\afrak=\bfrak$.
\end{corolario}

Using four strongly compact cardinals, Goldstern, Kellner and Shelah~\cite{GKS} applied Boolean ultrapowers (see~\cite{KTT}) to the poset that forces (\ref{leftGMS}) to prove the consistency of
\begin{equation}\label{maxGKS}
    \aleph_1<\add(\Nwf)<\cov(\Nwf)<\bfrak<\non(\Mwf)<\cov(\Mwf)<\dfrak<\non(\Nwf)<\cof(\Nwf)<\cfrak.
\end{equation}
With the same method, in~\cite{KST} Boolean ultrapowers of the poset that forces (\ref{leftKST}) guarantee the consistency of
\begin{equation}\label{maxKST}
    \aleph_1<\add(\Nwf)<\bfrak<\cov(\Nwf)<\non(\Mwf)<\cov(\Mwf)<\non(\Nwf)<\dfrak<\cof(\Nwf)<\cfrak.
\end{equation}
These results are examples of Cicho\'n's diagram divided into 10 different values (the maximum possible).

In this paper, we are also interested to get strengthenings or variations of (\ref{leftKST}), (\ref{leftGMS}), (\ref{maxGKS}) and (\ref{maxKST}) with respect to $\thzfc$ alone or with weaker large cardinal assumptions. The following result strengthens (\ref{leftGMS}) and solves~\cite[Question 7.1]{GMS}.

\begin{teorema}[Theorem~\ref{Apl1}]\label{thm7val}
    If $\theta_0\leq\theta_1\leq\theta_2\leq\mu\leq\nu$ are uncountable regular cardinals and $\lambda$ is a cardinal such that $\lambda^{{<}\theta_2}=\lambda$, then there is a ccc poset that forces (see Figure~\ref{Fig7val})
    \begin{multline*}
        \add(\Nwf)=\theta_0\leq\cov(\Nwf)=\theta_1\leq\bfrak=\afrak=\theta_2
        \leq\non(\Mwf)=\mu\\
        \leq\cov(\Mwf)=\nu\leq\dfrak=\non(\Nwf)=\cfrak=\lambda.
    \end{multline*}
\end{teorema}

\begin{remark}\label{remupdate}
   By the time this paper was submitted, Theorem~\ref{thm7val} was a new result in the sense that no large cardinals are used to prove it, and it is another example of Cicho\'n's diagram divided into 7 values, the maximum number of different values known consistent at that moment without using large cardinals. See more examples in~\cite{FFMM,mejiavert} for 7 values.

   Since the initial submission of this paper the consistency of~(\ref{maxGKS}) and~(\ref{maxKST}) was proved in~\cite{GKMS} without using large cardinals. Although  Theorems~\ref{thm7val},~\ref{allM},~\ref{allY} and~\ref{max3} are covered by this new result,
   our methods are different, and in particular Theorem~\ref{thm7val} may be useful as a starting point for further separation results for Cicho\'n's diagram.
   A more detailed discussion is provided in Section~\ref{SecDisc}.
\end{remark}

\begin{figure}
\begin{center}
  \includegraphics{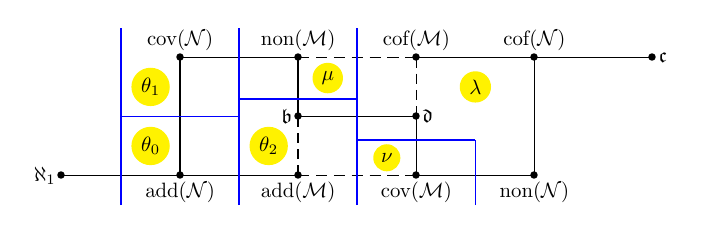}
  \caption{Seven values in Cicho\'n's diagram with the left side separated.}
  \label{Fig7val}
\end{center}
\end{figure}

The method to prove Theorem~\ref{thm7val} is a modification of the method in~\cite{GMS} to prove (\ref{leftGMS}), which is reviewed as follows. To force $\bfrak=\theta_2<\non(\Mwf)=\mu<\cov(\Mwf)=\cfrak=\lambda$, the idea is to perform a FS iteration of Suslin ccc posets restricted to small models, this to guarantee that each cardinal invariant of the left gets its desired value. Though classical techniques from~\cite{JS,Br} can be used, the main issue is to guarantee that $\bfrak$ does not get larger than desired. The reason is that restrictions of $\Eor$ are used (to increase $\non(\Mwf)$) along the iteration, and such restrictions may add dominating reals by a result of Pawlikowski~\cite{Paw-Dom}. Hence, chains of ultrafilters on $\omega$ are used to guarantee that no dominating reals are added, even more, to guarantee that the iteration is $\theta_2$-$\Fr$-Knaster. To achieve this, the following is required.

\begin{enumerate}[({P}1)]
    \item $2^{\theta_2}\geq\lambda$, so that at most $\theta_2$-sequences of ultrafilters are enough (by~\cite{TopThm}).
    \item $\theta^{\aleph_0}<\mu$ for any $\theta<\mu$.
    \item The chains of ultrafilters and the iteration are constructed simultaneously by recursion.
\end{enumerate}

Now, to prove Theorem~\ref{thm7val} we need to additionally separate $\cov(\Mwf)$ and $\dfrak$, which lie on the right side of Cicho\'n's diagram, while separating all the left side. The third author~\cite{M} has showed that Blass's and Shelah's~\cite{B1S} method of matrix iterations works to separate several cardinals on the left and right side simultaneously, which we use to produce a method of matrix iterations with matrices of ultrafilters to extend the method from~\cite{GMS}. Concretely, we introduce the concept of \emph{${<}\kappa$-uf-extendable matrix iteration} (see Definition~\ref{Defufextmatrix}) and prove the following result, which is considered the main result of this paper.

\begin{teorema}[Theorem~\ref{mainpres}]\label{uf-extKnaster}
    If $\kappa$ is an uncountable regular cardinal then every ${<}\kappa$-uf-extendable matrix iteration is $\kappa$-$\Fr$-Knaster.
\end{teorema}

In order to define this type of matrix iterations, we required to generalize the notion of $\Fr$-linked as follows. When $F$ is a free filter on $\omega$, $\Por$ is a poset and $Q\subseteq\Por$, we say that $Q$ is \emph{$F$-linked} if, for any sequence $\la p_n:n<\omega\ra$ of members of $Q$, there is some $q\in \Por$ forcing that $\{n<\omega:p_n\in\dot{G}\}$ is $F$-positive (note that this is the same as $\Fr$-linked when $F$ is the Frechet filter). In the natural way, the notions \emph{$\mu$-$F$-linked} and \emph{$\theta$-$F$-Knaster} are defined for posets. We also say that $Q\subseteq\Por$ is \emph{uf-linked} if $Q$ is $F$-linked for every free filter $F$ (equivalently, for every non-principal ultrafilter); say that $\Por$ is \emph{$\mu$-uf-linked} if it is the union of ${\leq}\mu$-many uf-linked subsets; and $\Por$ is \emph{$\theta$-uf-Knaster} if every subset of $\Por$ of size $\theta$ contains a uf-linked subset of the same size.\footnote{In general, these notions are not equivalent to ``$\mu$-$F$-linked (resp. $\theta$-$F$-Knaster) for every free filter $F$ on $\omega$".}

A curious fact proved in~\cite[Lemma~5.5]{mejiavert} (see Lemma~\ref{quasiuf}) is that, for ccc posets, the notions $\Fr$-linked and uf-linked are equivalent, which means that the notions above are not generalizations in the context of ccc. However, the notion $\mu$-uf-linked (for $\mu<\theta_2$) is implicitly used to construct the chains of ultrafilters in~\cite{GMS}, and it is also necessary to construct matrices of ultrafilters along an uf-extendable matrix iteration. For short, a ${<}\kappa$-uf extendable matrix iteration produces a FS iteration $\la\Por_\alpha,\Qor_\alpha:\alpha<\pi\ra$ (at the top of the matrix) of $\kappa$-cc posets where each iterand $\Qnm_\alpha$ is $\mu_\alpha$-uf-linked \emph{with respect} to a complete subposet of $\Por_\alpha$ (lying below in the matrix) for some $\mu_\alpha<\kappa$ (but not necessarily $\mu_\alpha$-uf-linked with respect to $\Por_\alpha$).

The most surprising thing about our method is that it does not rely on conditions like (P1)-(P3), e.g., the matrix iteration can be defined before considering any matrix of ultrafilters, and no restriction on the amount of matrices of ultrafilters is required. For each quite uniform countable $\Delta$-system $\la p_n:n<\omega\ra$ we can construct a matrix of ultrafilters along the matrix iteration and a condition $q$ forcing that $\{n<\omega:p_n\in\dot{G}\}$ is infinite, which will be enough to guarantee that the construction is $\kappa$-$\Fr$-Knaster.

The following constellation can also be proved by our method.

\begin{teorema}[Theorem~\ref{Apl2}]\label{allM}
    If $\theta_0\leq\theta_1\leq\mu\leq\nu$ are uncountable regular cardinals and $\lambda$ is a cardinal such that $\lambda^{{<}\theta_1}=\lambda$, then there is a ccc poset that forces (see Figure~\ref{FigM})
    \begin{multline*}
        \add(\Nwf)=\theta_0\leq\bfrak=\afrak=\theta_1\leq\cov(\Nwf)=\non(\Mwf)=\mu\\
        \leq\cov(\Mwf)=\non(\Nwf)=\nu\leq\dfrak=\cfrak=\lambda.
    \end{multline*}
\end{teorema}

\begin{figure}
\begin{center}
  \includegraphics{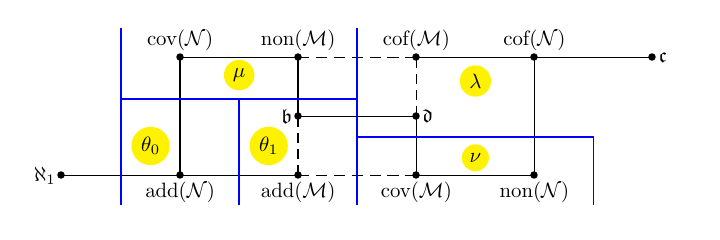}
  \caption{Separation of the cardinals associated with $\Mwf$ and $\Nwf$.}
  \label{FigM}
\end{center}
\end{figure}

Theorems~\ref{thm7val} and~\ref{allM} show (without using large cardinals) that the four cardinals $\add(\Mwf)$, $\cov(\Mwf)$, $\non(\Mwf)$ and $\cof(\Mwf)$ can consistently be pairwise different, which solves~\cite[Question 6.2]{CM}. In the latter theorem, the cardinals associated with $\Nwf$ are also forced to be pairwise different, though the consistency for $\Nwf$ alone was already proved in~\cite{M}.

As a consequence of Theorem~\ref{allM}, we are also able to solve a consistency problem related to \emph{Yorioka ideals}. Yorioka~\cite{Yorioka} defined $\sigma$-ideals  $\Ical_f$ parametrized by increasing functions $f\in\omega^\omega$ (see details in Definition~\ref{DefYorio}), to show that no inequality between $\cof(\SNcal)$ and $\cfrak:=2^\omega$ can be decided in $\thzfc$, where  $\mathcal{SN}$ is the $\sigma$-ideal of \emph{strong measure zero subsets of $2^\omega$}. The second and third author~\cite{CM} produced a ccc poset, via a matrix iteration, that forces the four cardinals $\add(\Iwf_f)$, $\cov(\Iwf_f)$, $\non(\Iwf_f)$ and $\cof(\Iwf_f)$ to be pairwise different for any $f$ above some fixed $f^*$. Now, thanks to Theorem~\ref{allM}, we can prove this consistency result \emph{for any $f$}, which solves~\cite[Question 6.1]{CM}.

\begin{teorema}[Corollary~\ref{Apl3}]\label{allY}
    There is a ccc poset that forces $\add(\Iwf_f)<\cov(\Iwf_f)<\non(\Iwf_f)<\cof(\Iwf_f)$ for any increasing $f\in\omega^\omega$.
\end{teorema}

As the final result, Boolean ultrapowers can be applied to the poset constructed for Theorem~\ref{thm7val} to weaken the large cardinal hypothesis of (\ref{maxGKS}).

\begin{teorema}[Theorem~\ref{cichonmax3}]\label{max3}
    Assuming three strongly compact cardinals, there is a ccc poset that forces
    \[
    \aleph_1<\add(\Nwf)<\cov(\Nwf)<\bfrak<\non(\Mwf)<\cov(\Mwf)<\dfrak<\non(\Nwf)<\cof(\Nwf)<\cfrak.
    \]
\end{teorema}

Result (\ref{maxGKS}) require further hypotheses, for example, GCH is assumed in the ground model, the cardinals on the left side of Cicho\'n's diagram cannot be successors of cardinals of countable cofinality, and the value for $\bfrak$ should be a successor. These assumptions can be omitted for Theorem~\ref{max3} except of GCH that can be weakened substantially.

This work is structured as follows.
\begin{enumerate}[\bfseries{Section }1.]
\setcounter{enumi}{1}
    \item We review our forcing notation, matrix iterations, and the classical preservation results for FS iterations and matrix iterations.
    \item We introduce Filter-linkedness, present examples of $\mu$-uf-posets and show Theorem~\ref{Kmad}. In addition, we show how $\mu$-$F$-linkedness and $\theta$-$F$-Knaster behaves in FS iterations and FS products.
    \item We define the notion of ${<}\kappa$-uf-extendable matrix iterations and prove Theorem~\ref{uf-extKnaster}.
    \item We show applications of Theorem~\ref{uf-extKnaster}, concretely, we prove Theorems~\ref{thm7val},~\ref{allM},~\ref{allY}, and~\ref{max3}.
    \item We discuss some open questions and recent updates related to this research.
\end{enumerate}

\section{Preliminaries}\label{SecPre}

We first review some notation and the forcing posets we are going to use throughout this paper. Denote $\Cor:=\omega^{<\omega}$ (Cohen forcing) ordered by $\supseteq$ and, for any set  $Z$, $\Cor_Z$ denotes the FS product of $\Cor$ along $Z$.

\begin{definition}[{\cite{BrF,Hechlermad}}]\label{DefHz}
    Let $Z$ be a set. \emph{Hechler's poset $\Hor_Z$ for adding an a.d.\ family (indexed by $Z$)} is defined as the poset whose conditions are of the form $p:F_p\times n_p\to2$ with $F_p\in[Z]^{<\aleph_0}$ and $n_p<\omega$ (demand $n_p=0$ iff $F_p=\emptyset$), ordered by $q\leq p$ iff $p\subseteq q$ and $|q^{-1}[\{1\}]\cap(F_p\times\{i\})|\leq 1$ for every $i\in[n_p,n_q)$.
\end{definition}

 This poset has the Knaster property, it even has precaliber $\aleph_1$, and the a.d.\ family it adds is maximal when $Z$ is uncountable.  For any $Z\subseteq Z'$, $\Hor_Z\lessdot\Hor_{Z'}$; and both $\Hor_Z$ and $\Hor_{Z'}$ are isomorphic whenever $|Z|=|Z'|$. The forcing $H_Z$ is forcing equivalent to $\Cor$ when $Z$ is countable and non-empty, and to $\Cor_{\omega_1}$ when $|Z|=\aleph_1$.\footnote{It is not hard to see that $\Hor_{\omega_1}$ is the FS iteration of the quotients $\Hor_{\alpha+1}/\Hor_\alpha$ for $\alpha<\omega_1$. Since these quotients are countable, $\Hor_{\omega_1}$ is equivalent to the FS support iteration of length $\omega_1$ of $\Cor$, which is $\Cor_{\omega_1}$.}

Though random forcing is typically known as the cBa (complete Boolean algebra) $\Bwf(2^\omega)/\Nwf$ (where $\Bwf(2^\omega)$ denotes the family of Borel subsets of $2^\omega$), we often use the following equivalent poset.

\begin{definition}\label{Defrandom}
       \emph{Random forcing} $\Bor$ is defined as the set of trees $T\subseteq2^{<\omega}$ such that $\lambda^*([T])>0$ where $\lambda^*$ denotes the Lebesgue measure on $2^\omega$. The order is $\subseteq$.

       For $(s,m)\in2^{<\omega}\times\omega$ set
       \[\Bor(s,m):=\{T\in\Bor: [T]\subseteq[s]\text{\ and }2^{|s|}\cdot\lambda^*([T])\geq 1-2^{-10-m}\}.\]
\end{definition}

Note that, for fixed $m<\omega$, $\bigcup_{s\in 2^{{<}\omega}}\Bor(s,m)$ is dense in $\Bor$.

\begin{notation}\label{DefSlalom}
   \begin{enumerate}[(1)]
       \item A \emph{slalom} is a function $\varphi:\omega\to[\omega]^{<\omega}$. For any function $x$ with domain $\omega$, $x\in^*\varphi$ denotes $\forall^\infty i<\omega(x(i)\in\varphi(i))$, which is read \emph{$\varphi$ localizes $x$}.
       \item For a function $b$ with domain $\omega$ and $h\in\omega^\omega$, denote $\seq_{<\omega}(b):=\bigcup_{n<\omega}\prod_{i<n}b(i)$, $\prod b:=\prod_{i<\omega}b(i)$ and $\Swf(b,h):=\prod_{i<\omega}[b(i)]^{\leq h(i)}$.
       \item For any set $A$, $\id_A$ denotes the identity function on $A$. Denote $\id:=\id_\omega$.
       \item Operations and relations between functions from $\omega$ into the ordinals are interpreted pointwise. For example, if $b$ and $c$ are such functions, $b\cdot c$ denotes the pairwise ordinal product of both functions, and $b<c$ indicates that $b(i)<c(i)$ for any $i<\omega$. Also, constant objects may be interpreted as constant functions with domain $\omega$, for instance, the $\omega$ in $\Swf(\omega,h)$ is understood as the constant function $\omega$.
       \item For $x,y:\omega\to\mathrm{On}$, $x\leq^* y$ denotes $\forall^\infty i<\omega(x(i)\leq y(i))$, which is read \emph{$x$ is dominated by $y$}. Likewise, $x<^* y$ is defined.
       \item Say that two functions $x$ and $y$ with domain $\omega$ are \emph{eventually different}, denoted by $x\neq^* y$, if $\forall^\infty i<\omega(x(i)\neq y(i))$.
   \end{enumerate}
\end{notation}

\begin{definition}\label{DefLocforcing}
   \emph{Localization forcing} is the poset
   \[\Loc:=\{\varphi\in\Swf(\omega,\id):\exists m<\omega\forall i<\omega(|\varphi(i)|\leq m)\}\]
   ordered by $\varphi'\leq\varphi$ iff $\varphi(i)\subseteq\varphi'(i)$ for every $i<\omega$. Recall that this poset is $\sigma$-linked and that it adds a slalom in $\Scal(\omega,\id)$ that localizes all the ground model reals in $\omega^\omega$.
\end{definition}

The following is a generalization of the standard ccc poset that adds an eventually different real (see e.g.~\cite{KO,CM}).

\begin{definition}[ED forcing]\label{DefEDforcing}
  Fix $b:\omega\to\omega+1\menos\{0\}$ and $h\in\omega^\omega$ such that $\lim_{i\to+\infty}\frac{h(i)}{b(i)}=0$ (when $b(i)=\omega$, interpret $\frac{h(i)}{b(i)}$ as $0$).
      Define the \emph{$(b,h)$-ED (eventually different real) forcing} $\Eor^h_b$ as the poset whose conditions are of the form $p=(s,\varphi)$ such that, for some $m:=m_{p}<\omega$, 
         \begin{enumerate}[(i)]
            \item $s\in\seq_{<\omega}(b)$, $\varphi\in\Swf(b,m\cdot h)$, and
            \item $m\cdot h(i)<b(i)$ for every $i\geq|s|$,
         \end{enumerate}
      ordered by $(t,\psi)\leq(s,\varphi)$ iff $s\subseteq t$, $\forall i<\omega(\varphi(i)\subseteq\psi(i))$ and $t(i)\notin\varphi(i)$ for all $i\in|t|\menos|s|$.

      Put $\Eor^h_b(s,m):=\{(t,\varphi)\in\Eor^h_b:t=s\text{\ and }m_{(t,\varphi)}\leq m\}$ for $s\in\seq_{<\omega}(b)$ and $m<\omega$.

      Denote $\Eor_b:=\Eor^1_b$, $\Eor:=\Eor_\omega$, $\Eor_b(s,m):=\Eor^1_b(s,m)$, and $\Eor(s,m):=\Eor_\omega(s,m)$.

     It is not hard to see that $\Eor_b^h$ is $\sigma$-linked. Even more, whenever $b\geq^*\omega$, $\Eor_b^h$ is $\sigma$-centered.

     When $h\geq^*1$, $\Eor_b^h$ adds an eventually different real\footnote{More generally, $\Eor_b^h$ adds a real $e\in\prod b$ such that $\forall^\infty_{i<\omega}(e(i)\notin\varphi(i))$ for any $\varphi\in\Scal(b,h)$ in the ground model.} in $\prod b$.
\end{definition}

One of our main results deals with the following notion.

\begin{definition}[Yorioka {\cite{Yorioka}}]\label{DefYorio}
   For $\sigma \in (2^{<\omega})^{\omega}$ define
\[[\sigma]_\infty:=\{x \in 2^{\omega}:\exists^{\infty}{n < \omega}^{}(\sigma(n) \subseteq x)\}=\bigcap_{n<\omega} \bigcup_{m \geqslant n}[\sigma(m)]\]
and $\hgt_{\sigma}\in\omega^{\omega}$ by $\hgt_{\sigma}(i):=|\sigma(i)|$ for each $i<\omega$.

Define the relation $\ll$ on $\omega^\omega$ by
$f\ll g$  iff $\forall{k<\omega}\forall^{\infty}{n<\omega}(f(n^k)\leq g(n))$.

For $g,f \in \omega^\omega$ define the families
    \[\mathcal{J}_{g}:=\{X\subseteq 2^{\omega}:\exists{\sigma \in (2^{<\omega})^{\omega}}(X \subseteq [\sigma]_\infty\text{\ and }\hgt_{\sigma}=g )\}\textrm{\ and }
    \mathcal{I}_{f}:=\bigcup_{g\gg f}\mathcal{J}_{g}.\]
 Any family of the form $\mathcal{I}_{f}$ with $f$ increasing is called a \textit{Yorioka ideal}.
\end{definition}

The following results show the relationship between the cardinal invariants associated with Yorioka ideals and the cardinals in Cicho\'n's diagram.
This is used in Section~\ref{SecAppl} to prove  Corollary~\ref{Apl3}. See Figure~\ref{FigI_f} on page~\pageref{FigI_f} for a diagram of some of these inequalities.

\begin{theorem}\label{YorioCichdiagram}
Let $f \in \omega^\omega$ be a strictly increasing function. Then
\begin{enumerate}[(a)]
\item \emph{(Yorioka~\cite{Yorioka})} $\mathcal{I}_{f}$ is a $\sigma$-ideal and $\mathcal{SN}\subseteq \mathcal{I}_{f}\subseteq\mathcal{N}$, so $\cov(\mathcal{N})\leq \cov(\mathcal{I}_f)\leq \cov(\mathcal{SN})$ and $\non(\mathcal{SN})\leq \non(\mathcal{I}_f)\leq \non(\mathcal{N})$.
\item \emph{(Kamo, see e.g.~\cite[Cor.~3.14]{CM})} $\add(\Ncal)\leq\add(\Ical_f)$ and $\cof(\Ical_f)\leq\cof(\Ncal)$.
\item \emph{(Kamo and Osuga~\cite{kamo-osuga})} $\add(\mathcal{I}_f)\leq \bfrak$ and $\dfrak\leq \cof(\mathcal{I}_f)$.
\item \emph{(Osuga~\cite{O1}, see also~\cite[Cor.~3.22]{CM})} $\cov(\Ical_f)\leq \non(\Mcal)$ and $\cov(\Mcal)\leq \non(\mathcal{I}_f)$.
\end{enumerate}
\end{theorem}

\subsection{Simple matrix iterations}

We review some usual facts about matrix iterations in the context of what we call \emph{simple matrix iterations}. In this type of matrix iterations only restricted generic reals are added, and preservation properties behave very nicely.

\begin{definition}\label{DefCompM}
Let $M$ be a transitive model of $\thzfc$ (or of a finite fragment of it). Given two posets $\Por\in M$  and $\Qor$ (not necessarily in $M$), say that \textit{$\Por$ is a \textit{complete subposet of $\Qor$} with respect to $M$}, denoted by $\Por \lessdot_{M}\Qor$, if $\Por$ is a subposet of $\Qor$ and every maximal antichain in $\Por$ that belongs to $M$ is also a maximal antichain in $\Qor$.
\end{definition}

In this case, if $N$ is another transitive model of $\thzfc$ such that $N\supseteq M$ and $\Qor\in N$, then $\Por\lessdot_{M} \Qor$ implies that, whenever $G$ is $\Qor$-generic over $N$, $G \cap \Por$ is $\Por$-generic over $M$ and $M[G \cap \Por]\subseteq N[G]$ (see Figure~\ref{Figsinglestep}). When $\Por\in M$ it is clear that $\Por\lessdot_M\Por$.

\begin{figure}
\begin{center}
  \includegraphics{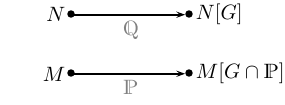}
  \caption{Generic extensions of pairs of posets ordered like $\Por\lessdot_M\Qor$.}
  \label{Figsinglestep}
\end{center}
\end{figure}

\begin{definition}[Blass and Shelah {\cite{B1S}}]\label{Defmatit}
A \textit{matrix iteration} $\mbf$ consists of
\begin{enumerate}[(I)]
  \item a well order $I^\mbf$ and an ordinal $\pi^\mbf$,
  \item for each $i\in I^\mbf$, a FS iteration $\Pbb^\mbf_{i,\pi^\mbf}=\la\Pbb^\mbf_{i,\xi},\Qnm^\mbf_{i,\xi}:\xi<\pi^\mbf\ra$ such that, for any $i\leq j$ in $I^\mbf$ and $\xi<\pi^\mbf$, if $\Pbb^\mbf_{i,\xi}\lessdot\Pbb^\mbf_{j,\xi}$ then $\Pbb^\mbf_{j,\xi}$ forces $\Qnm^\mbf_{i,\xi}\lessdot_{V^{\Pbb^\mbf_{i,\xi}}}\Qnm^\mbf_{j,\xi}$.
\end{enumerate}
According to this notation, $\Pbb^\mbf_{i,0}$ is the trivial poset and $\Pbb^\mbf_{i,1}=\Qnm^\mbf_{i,0}$.
By Lemma~\ref{parallellim}, $\Pbb^\mbf_{i,\xi}$ is a complete subposet of $\Pbb^\mbf_{j,\xi}$ for all $i\leq j$ in $I^\mbf$ and $\xi\leq\pi^\mbf$.

We drop the upper index $\mbf$ when it is clear from the context. If $j\in I$ and $G$ is $\mathbb{P}_{j,\pi}$-generic over $V$ we denote  $V_{i,\xi}=V[G\cap\Pbb_{i,\xi}]$ for all $i\leq j$ in $I$ and $\xi\leq\pi$ . Clearly, $V_{i,\xi}\subseteq V_{j,\eta}$ for all $i\leq j$ in $I$ and $\xi\leq\eta\leq\pi$. The idea of such a construction is to obtain a matrix $\langle V_{i,\xi}:i\in I,\ \xi\leq\pi\rangle$ of generic extensions as illustrated in Figure~\ref{FigMatit}.

\begin{figure}
\begin{center}
  \includegraphics[width=\textwidth]{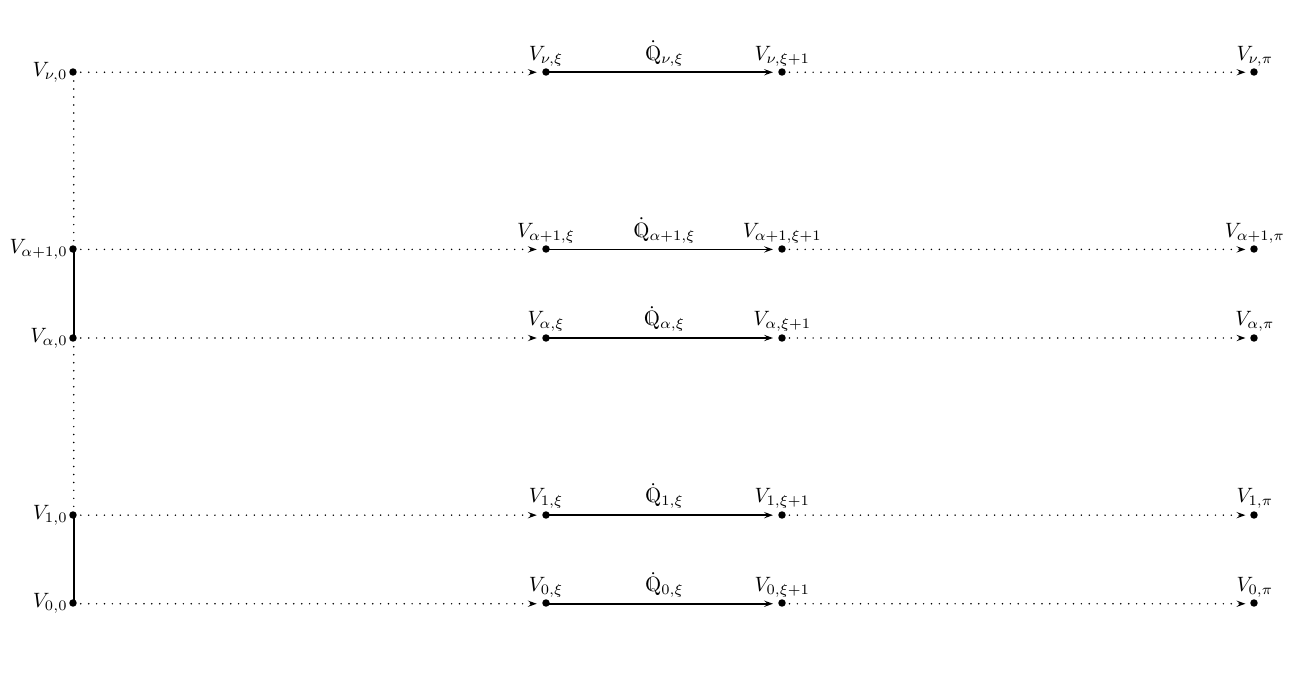}
  \caption{Matrix iteration with $I^\mbf=\nu+1$ where $\nu$ is an ordinal.}
  \label{FigMatit}
\end{center}
\end{figure}

If $\xi\leq\pi$, $\mbf\frestr\xi$ (\emph{horizontal restriction}) denotes the matrix iteration with $I^{\mbf{\upharpoonright}\xi}=I$ and $\pi^{\mbf{\upharpoonright}\xi}=\xi$ where the FS iterations are the same as in (II) but restricted to $\xi$. On the other hand, for any $J\subseteq I$, $\mbf|J$ (\emph{vertical restriction}) denotes the matrix iteration with $I^{\mbf|J}=J$ and $\pi^{\mbf|J}=\pi$ where the FS iterations for $i\in J$ are exactly as in (II).
\end{definition}

Although $I^\mbf$ is an ordinal in all our applications, it is more practical to use it as a well order in general because it eases the notation when dealing with $\mbf|J$ in the case that $J$ is a set of ordinals but not an ordinal (as in the last part of the proof of Main Lemma~\ref{main}).

The following type of matrix iteration is the one we are going to deal with throughout the whole text.

\begin{definition}[Simple matrix iteration]\label{Defcohsys}
   A \emph{simple matrix iteration} $\mbf$ is a matrix iteration, composed additionally of
   a function $\Delta^\mbf:\pi^\mbf\to I^\mbf$, that satisfies: for each $\xi<\pi^\mbf$, there is a $\Por^\mbf_{\Delta^\mbf(\xi),\xi}$-name $\Qnm^\mbf_\xi$ of a poset such that, for each $i\in I^\mbf$,
         \[\Qnm^\mbf_{i,\xi}=\left\{\begin{array}{ll}
                 \Qnm^\mbf_\xi  & \text{if $i\geq\Delta^\mbf(\xi)$,}\\
                 \mathds{1} & \text{otherwise.}
                 \end{array}\right.\]
   The upper index $\mbf$ is omitted when there is
   no risk of ambiguity.
\end{definition}

A simple matrix iteration is easily constructed by recursion on $\xi\leq\pi$. When $\mbf\frestr\xi$ is already constructed, $\Delta(\xi)$ and $\Qnm_\xi$ are freely defined, which allows to extend the matrix to $\mbf\frestr(\xi+1)$. Limit steps are uniquely determined by taking direct limits. 
When $\Qnm_\xi$ adds a real, it will be generic over $V_{\Delta(\xi),\xi}$ but not necessarily over $V_{i,\xi}$ for larger $i$, which is the reason we say that a \emph{restricted generic real} is added at step $\xi$.
For instance, let $\Dor$ be the Hechler poset for adding a dominating real. When $\Qnm_\xi=\Dor^{V_{\Delta(\xi),\xi}}$, the generic real added at $\xi$ is dominating over $V_{\Delta(\xi),\xi}$. Moreover, more restricted generic sets are allowed, for example, when $\Qnm_\xi=\Dor^{N_\xi}$ where $N_\xi\in V_{\Delta(\xi),\xi}$ is a (small) transitive model of $\thzfc$, the generic real added at step $\xi$ is dominating over $N_\xi$ but not necessarily over $V_{\Delta(\xi),\xi}$.

Most of the time we deal with simple matrix iterations where $I^\mbf=\nu+1$ for some ordinal $\nu$, unless we are reasoning with restrictions of such matrix iteration. In this case, if the simple matrix iteration is composed by ccc posets and $\nu$ has uncountable cofinality, then $\Por_{\nu,\xi}$ is the direct limit of the posets below it in the matrix. More generally:

\begin{lemma}[{\cite{BrF}}, see also {\cite[Cor.~2.6]{mejiavert}}]\label{realint}
Let $\theta$ be an uncountable regular cardinal and let $\nu$ be an ordinal.  Assume that $\mbf$ is a simple matrix iteration such that
\begin{enumerate}[(i)]
   \item $I^\mbf=\nu+1$, $\cf(\nu)\geq\theta$,
   \item $\nu\notin\ran\Delta$, and
   \item for each $\xi<\pi^\mbf$, $\Por_{\nu,\xi}$ forces that $\Qnm_{\Delta(\xi),\xi}$ is $\theta$-cc.
\end{enumerate}
Then, for any $\xi\leq\pi$,
\begin{enumerate}
    \item[(a)]  $\Por_{\nu,\xi}$ is the direct limit of $\la\Por_{\alpha,\xi}:\alpha<\nu\ra$, and
    \item[(b)] if $\beta<\theta$ and $\dot{f}$ is a $\Por_{\nu,\xi}$-name of a function from $\beta$ into $\bigcup_{\alpha<\nu}V_{\alpha,\xi}$ then $\dot{f}$ is forced to be equal to a $\Por_{\alpha,\xi}$-name for some $\alpha<\nu$.  In particular, the reals in $V_{\nu,\xi}$ are precisely the reals in $\bigcup_{\alpha<\nu}V_{\alpha,\xi}$.
\end{enumerate}
\end{lemma}

\subsection{Relational systems and preservation}\label{SubsecRelSys}

We review the theory of preservation properties for FS iterations developed by Judah and Shelah~\cite{JS} and Brendle~\cite{Br}. A similar presentation also appears in~\cite{mejiavert} and a generalized version can be found in~\cite[Sect.~4]{CM}.

\begin{definition}\label{DefRelSys}
A \textit{relational system} is a triple $\Rbf=\la X, Y, \sqsubset\ra$ where $\sqsubset$ is a relation.\footnote{Although the relation $\sqsubset$ is only relevant when restricted to $X\times Y$, there is no need to demand it to be contained in $X\times Y$. See Example~\ref{ExmPrs}(5) where $\sqsubset$ is just $=$.} For $x\in X$ and $y\in Y$, $x\sqsubset y$ is often read $y\sqsubset$-\textit{dominates} $x$.
\begin{enumerate}[(1)]
    \item A family $F\subseteq X$ is \textit{$\Rbf$-bounded} if there is a member of $Y$ that $\sqsubset$-dominates every member of $F$, otherwise we say that the set is \emph{$\Rbf$-unbounded}. Define the cardinal $\bfrak(\Rbf)$ as the smallest size of an $\Rbf$-unbounded family.

    \item Dually, $D\subseteq Y$ is \textit{$\Rbf$-dominating} if every member of $X$ is $\sqsubset$-dominated by some member of $D$. The cardinal $\dfrak(\Rbf)$ is defined as the smallest size of an $\Rbf$-dominating family.

    \item An object $x\in X$ is \textit{$\Rbf$-unbounded over a set $M$} if $x\not\sqsubset y$ for all $y\in Y\cap M$.

    \item If $\theta$ is a cardinal, a family $F\subseteq X$ is \emph{strongly $\theta$-$\Rbf$-unbounded} if $|F|\geq\theta$ and, for any $y\in Y$, $|\{x\in F : x\sqsubset y\}|<\theta$.
\end{enumerate}
\end{definition}

\begin{lemma}\label{StrUnb-effect}
   In the context of Definition~\ref{DefRelSys}, assume that $F\subseteq X$ is a strongly $\theta$-$\Rbf$-unbounded family. Then:
   \begin{enumerate}[(a)]
       \item $F$ is $\Rbf$-unbounded, in particular, $\bfrak(\Rbf)\leq|F|$.\footnote{The case $\theta\leq1$ is uninteresting and trivial. When $\theta=0$, the existence of a strongly $\theta$-$\Rbf$-unbounded family implies $Y=\emptyset$, so $\bfrak(\Rbf)=0$; when $F$ is a strongly $1$-$\Rbf$-unbounded family, $F\neq\emptyset$ and, for any $y\in Y$, no member of $F$ is $\sqsubset$-dominated by $y$, so $\bfrak(\Rbf)\leq1$.}
       \item Whenever $\theta$ is regular, $|F|\leq\dfrak(\Rbf)$.
   \end{enumerate}
\end{lemma}

The previous fact is the reason why strongly unbounded families are used to obtain upper bounds of $\bfrak(\Rbf)$ and lower bounds of $\dfrak(\Rbf)$, so their preservation in forcing extensions helps to force values for such cardinals.

The following two definitions are the central concepts for preservation of strongly unbounded families.

\begin{definition}
Say that $\Rbf=\langle X,Y,\sqsubset\rangle$ is a \textit{Polish relational system (Prs)} if the following is satisfied:
\begin{enumerate}[(i)]
\item $X$ is a perfect Polish space,
\item $Y$ is a non-empty analytic subspace of some Polish space $Z$ and
\item $\sqsubset\cap(X\times Z)=\bigcup_{n<\omega}\sqsubset_{n}$ where $\la\sqsubset_{n}\ra_{n<\omega}$  is some increasing sequence of closed subsets of $X\times Z$ such that, for any $n<\omega$ and for any $y\in Y$,
$(\sqsubset_{n})^{y}=\{x\in X:x\sqsubset_{n}y \}$ is closed nowhere dense.
\end{enumerate}

By (iii), $\la X,\Mcal(X),\in\ra$ is Tukey-Galois below $\Rbf$ where $\Mcal(X)$ denotes the $\sigma$-ideal of meager subsets of $X$. Therefore, $\bfrak(\Rbf)\leq \non(\Mcal)$ and $\cov(\Mcal)\leq\dfrak(\Rbf)$.
\end{definition}

\begin{definition}[Judah and Shelah {\cite{JS}}]\label{JS}
Let $\Rbf=\la X, Y, \sqsubset\ra$ be a Prs and let $\theta$ be a cardinal. A poset $\Por$ is \textit{$\theta$-$\Rbf$-good} if, for any $\Por$-name $\dot{h}$ for a member of $Y$, there is a non-empty $H\subseteq Y$ (in the ground model) of size ${<}\theta$ such that, for any $x\in X$, if $x$ is $\Rbf$-unbounded over  $H$ then $\Vdash x\not\not\sqsubset \dot{h}$.

Say that $\Por$ is \textit{$\Rbf$-good} if it is $\aleph_1$-$\Rbf$-good.
\end{definition}


Definition~\ref{JS} describes a property used to preserve strongly $\Rbf$-unbounded families, which is respected by FS iterations. Concretely, when $\theta$ is uncountable regular,

\begin{enumerate}[(I)]
    \item any $\theta$-$\Rbf$-good poset preserves all the strongly $\theta$-$\Rbf$-unbounded families from the ground model and
    \item FS iterations of $\theta$-cc $\theta$-$\Rbf$-good posets produce $\theta$-$\Rbf$-good posets.
\end{enumerate}

Hence, by Lemma~\ref{StrUnb-effect}, posets that are $\theta$-$\Rbf$-good work to preserve $\bfrak(\Rbf)$ small and $\dfrak(\Rbf)$ large.
Clearly, $\theta$-$\Rbf$-good implies $\theta'$-$\Rbf$-good whenever $\theta\leq \theta'$, and any poset completely embedded into a $\theta$-$\Rbf$-good poset is also $\theta$-$\Rbf$-good. Also note the trivial fact that any poset is $\dfrak(\Rbf)^+$-good.

Now, we present the instances of Prs and their corresponding good posets that we use in our applications.

\begin{lemma}[{\cite[Lemma~4]{M}}]
If $\Rbf$ is a Prs and $\theta$ is an uncountable regular cardinal then any poset of size ${<}\theta$ is $\theta$-$\Rbf$-good. In particular, Cohen forcing is $\Rbf$-good.
\end{lemma}

\begin{example}\label{ExmPrs}
  \begin{enumerate}[(1)]
     \item \textit{Preserving non-meager sets:} Consider the Polish relational system $\Ed:=\la\omega^\omega,\omega^\omega,\neq^*\ra$. By~\cite[Thm.~2.4.1 \& Thm.~2.4.7]{BJ},  $\bfrak(\Ed)=\non(\Mcal)$ and $\dfrak(\Ed)=\cov(\Mcal)$.

     \item \textit{Preserving unbounded families:} The relational system $\Dbf:=\la\omega^\omega,\omega^\omega,\leq^{*}\ra$ is Polish and $\bfrak(\Dbf)=\bfrak$ and $\dfrak(\Dbf)=\dfrak$.
     Any $\mu$-$\mathrm{Fr}$-linked poset is $\mu^+$-$\Dbf$-good (see Lemma~\ref{ufextGood}).

     \item \textit{Preserving null-covering families:} Define $\Omega_n:=\{a\in [2^{<\omega}]^{<\aleph_0}:\lambda^*(\bigcup_{s\in a}[s])\leq 2^{-n}\}$ (endowed with the discrete topology) and put $\Omega:=\prod_{n<\omega}\Omega_n$ with the product topology, which is a perfect Polish space. For every $x\in \Omega$ denote $N_{x}^{*}:=\bigcap_{n<\omega}\bigcup_{s\in x(n)}[s]$, which is clearly a Borel null set in $2^{\omega}$.

     Define the Prs $\Cn:=\la \Omega, 2^\omega, \sqsubset\ra$ where $x\sqsubset z$ iff $z\notin N_{x}^{*}$. Recall that any null set in $2^\omega$ is a subset of $N_{x}^{*}$ for some $x\in \Omega$, so $\Cn$ and $\la\Ncal(2^\omega), 2^\omega, \not\ni\ra$ are Tukey-Galois equivalent. Hence, $\bfrak(\Cn)=\cov(\Ncal)$ and $\dfrak(\Cn)=\non(\Ncal)$.

     Any $\mu$-centered poset is $\mu^+$-$\Cn$-good (see e.g.~\cite{Br}). In particular, $\sigma$-centered posets are $\Cn$-good.

     \item \textit{Preserving union of null sets is not null:} For each $k<\omega$ let $\id^k:\omega\to\omega$ such that $\id^k(i)=i^k$ for all $i<\omega$ and $\Hcal:=\{\id^{k+1}:k<\omega\}$. Let $\Lc^*:=\la\omega^\omega, \Scal(\omega, \Hcal), \in^*\ra$ be the Polish relational system where \[\Scal(\omega, \Hcal):=\{\varphi:\omega\to[\omega]^{<\aleph_0}:\exists{h\in\Hcal}\forall{i<\omega}(|\varphi(i)|\leq h(i))\}.\]
     As consequence of~\cite[Thm.~2.3.9]{BJ}, $\bfrak(\Lc^*)=\add(\Ncal)$ and $\dfrak(\Lc^*)=\cof(\Ncal)$.

     Any $\mu$-centered poset is $\mu^+$-$\Lc^*$-good (see~\cite{Br,JS}) so, in particular, $\sigma$-centered posets are $\Lc^*$-good. Besides,  Kamburelis~\cite{Ka} showed that any Boolean algebra with a strictly positive finitely additive measure is $\Lc^*$-good (in particular, subalgebras of random forcing).

     \item\label{ExmPresCont} \emph{Preserving large continuum:} Consider the Polish relational system $\Id:=\la\omega^\omega,\omega^\omega,=\ra$. It is clear that $\bfrak(\Id)=2$ and $\dfrak(\Id)=\cfrak$. Though this is a quite trivial Prs, we are interested in the following simple facts:
     \begin{enumerate}[({\ref{ExmPresCont}.}1)]
         \item $x\in\omega^\omega$ is $\Id$-unbounded over $M$ iff $x\notin M$.
         \item If $\theta\geq 2$ then $F\subseteq\omega^\omega$ is strongly $\theta$-$\Id$-unbounded iff $|F|\geq \theta$.
         \item Any $\theta$-cc poset is $\theta$-$\Id$-good.\footnote{The converse is true when $\theta\leq\cfrak$. On the other hand, any poset is $\cfrak^+$-$\Id$-good.}
     \end{enumerate}
     Concretely, we use (5.2) as a simple resource to justify why the continuum is increased after Boolean ultrapowers of a ccc poset (Theorem~\ref{cichonmax3}).
  \end{enumerate}
\end{example}

The following results indicate that strongly unbounded families can be added with Cohen reals, and the effect on $\bfrak(\Rbf)$ and $\dfrak(\Rbf)$ by a FS iteration of good posets.

\begin{lemma}
Let $\mu$ be a cardinal with uncountable cofinality, $\Rbf=\la X, Y, \sqsubset\ra$ a Prs and let $\la\mathbb{P}_{\alpha}\ra_{\alpha<\mu}$ be a $\lessdot$-increasing sequence of $\cf(\mu)$-cc posets such that  $\Por_\mu=\limdir_{\alpha<\mu}\Por_{\alpha}$. If $\Por_{\alpha+1}$ adds a Cohen real $\dot{c}_\alpha\in X$ over $V^{\Por_\alpha}$ for any $\alpha<\mu$, then $\Por_{\mu}$ forces that $\{\dot{c}_\alpha:\alpha<\mu\}$ is a strongly $\mu$-$\Rbf$-unbounded family of size $\mu$.
\end{lemma}

\begin{theorem}\label{sizeforbd}
Let $\theta$ be an uncountable regular cardinal, $\Rbf=\la X, Y, \sqsubset\ra$ a Prs, $\pi\geq\theta$ an ordinal, and let $\Por_{\pi}=\langle\Por_{\alpha},\Qnm_{\alpha}:\alpha<\pi\rangle$ be a FS iteration such that, for each $\alpha<\pi$, $\Qnm_{\alpha}$ is a $\Por_{\alpha}$-name of a  non-trivial $\theta$-$\Rbf$-good $\theta$-cc poset. Then, $\Por_{\pi}$ forces that $\bfrak(\Rbf)\leq\theta$ and  $|\pi|\leq\dfrak(\Rbf)$.
\end{theorem}
\begin{proof}
See e.g.~\cite[Thm.~4.15]{CM} or~\cite[Thm.~3.6]{GMS}.
\end{proof}

For the remainder of this section, fix transitive models $M\subseteq N$ of $\thzfc$ and a Prs $\Rbf=\langle X,Y,\sqsubset\rangle$ coded in $M$. The following results are related to preservation of $\Rbf$-unbounded reals along simple matrix iterations.

\begin{lemma}[{\cite[Lemma~11]{BrF}}, see also {\cite[Lemma~5.13]{mejia-temp}}]
   Assume that $\Por\in M$ is a poset. Then, in $N$,  $\Por$ forces that every $c\in X^N$ that is $\Rbf$-unbounded over $M$ is $\Rbf$-unbounded over $M^{\Por}$.
\end{lemma}

\begin{lemma}[{\cite{BrF}}]\label{parallellim}
   Assume that $\Por_{0,\pi}=\langle\Por_{0,\alpha},\Qnm_{0,\alpha}:\alpha<\pi\ra\in M$ and $\Por_{1,\pi}=\langle\Por_{1,\alpha},\Qnm_{1,\alpha}:\alpha<\pi\ra\in N$ are FS iterations such that, for any $\alpha<\pi$, if $\Por_{0,\alpha}\lessdot_M\Por_{1,\alpha}$ then $\Por_{1,\alpha}$ forces that $\Qnm_{0,\alpha}\lessdot_{M^{\Por_{0,\alpha}}} \Qnm_{1,\alpha}$. Then $\Por_{0,\alpha}\lessdot_M\Por_{1,\alpha}$ for any $\alpha\leq\pi$.

   In addition, if $\pi$ is limit, $c\in X^N$ and, for any $\alpha<\pi$, $\Por_{1,\alpha}$ forces (in $N$) that $c$ is $\Rbf$-unbounded over $M^{\Por_{0,\alpha}}$, then $\Por_{1,\pi}$ forces  that $c$ is $\Rbf$-unbounded over $M^{\Por_{0,\pi}}$.
\end{lemma}

\begin{theorem}[{\cite{BrF}}, see also {\cite[Thm.~10 \& Cor.~1]{M}}]\label{matsizebd}
   Let $\mbf$ be a simple matrix iteration, and let $\Rbf=\la X,Y,\sqsubset\ra$ be a Polish relational system coded in $V$. Assume that, for any $\alpha\in I$, there is some $\xi_\alpha<\pi$ such that $\Por_{\alpha+1,\xi_\alpha}$ adds a real $\dot{c}_\alpha\in X$ that is $\Rbf$-unbounded over $V_{\alpha,\xi_\alpha}$. (Here, $\alpha+1$ denotes the immediate successor of $\alpha$ in $I$.)
   Then, for any $\alpha\in I$, $\Por_{\alpha+1,\pi}$ forces that $\dot{c}_{\alpha}$ is $\Rbf$-unbounded over $V_{\alpha,\pi}$.

   In addition, if $\mbf$ satisfies the hypothesis of Lemma~\ref{realint} with $\nu$ a cardinal of uncountable cofinality and $\theta=\cf(\nu)$, and $f:\cf(\nu)\to\nu$ is increasing and cofinal, then  $\Por_{\nu,\pi}$ forces that $\{\dot{c}_{f(\zeta)}:\zeta<\cf(\nu)\}$ is a strongly $\cf(\nu)$-$\Rbf$-unbounded family.
\end{theorem}

\section{Filter-linkedness}\label{Secufposets}

We first review some notation about filters on $\omega$. Denote by $\Fr:=\{x\subseteq\omega:|\omega\menos x|<\aleph_0\}$ the \emph{Frechet filter}. A filter $F$ on $\omega$ is \emph{free} if $\Fr\subseteq F$. A set $x\subseteq\omega$ is \emph{$F$-positive} if it intersects every member of $F$. Denote by $F^+$ the family of $F$-positive sets. Note that $x\in\Fr^+$ iff $x$ is an infinite subset of $\omega$.

In this section, we generalize the notion of \emph{Frechet-linkedness} from~\cite{mejiavert} (corresponding to $\Fr$-linkedness in Definition~\ref{Defufext}(2)), and its corresponding notions of linkedness and Knaster for posets. The third author~\cite{mejiavert} showed that $\mu$-$\Fr$-linked posets are $\mu^+$-$\Dbf$-good, so they preserve strongly $\mu^+$-$\Dbf$-unbounded families from the ground model, and they also satisfy a strong property related to preservation of certain type of mad families. We also look at the corresponding Knaster property and show that it implies preservation of strongly $\Dbf$-unbounded families and of the type of mad families presented in Definition~\ref{DefMd}. At the end, we show how these notions behave in FS iterations and FS products.

\subsection{Filter-linkedness and examples}

\begin{definition}\label{Defufext}
Let $\Por$ be a poset, $F$ a free filter on $\omega$ and let $\mu$ be an infinite cardinal.
\begin{enumerate}[(1)]
    \item The $\Por$-name $\dot{G}$ usually denotes the canonical name of the $\Por$ generic set. If $\bar{p}=\la p_n:n<\omega\ra$ is a sequence in $\Por$, denote by $\dot{W}_\Por(\bar{p})$ the $\Por$-name of $\{n<\omega:p_n\in\dot{G}\}$. When the forcing is understood from the context, we just write $\dot{W}(\bar{p})$.
    \item A set $Q\subseteq\Por$ is \emph{$F$-linked} if, for any sequence $\bar{p}=\la p_n:n<\omega\ra$ in $Q$, there exists a $q\in\Por$ that forces $\dot{W}(\bar{p})\in F^+$.
    \item A set $Q\subseteq\Por$ is \emph{ultrafilter-linked}, abbreviated \emph{uf-linked}, if $Q$ is $D$-linked for any non-principal ultrafilter $D$ on $\omega$.
    \item The poset $\Por$ is \emph{$\mu$-$F$-linked} if $\Por=\bigcup_{\alpha<\mu}P_\alpha$ for some sequence $\la P_\alpha:\alpha<\mu\ra$ of $F$-linked subsets of $\Por$.

    When each $P_\alpha$ is uf-linked, we say that
    $\Por$ is \emph{$\mu$-uf-linked}.

    When $\mu=\aleph_0$, we write \emph{$\sigma$-$F$-linked} and \emph{$\sigma$-uf-linked}.
    \item When $\kappa$ is an uncountable cardinal, say that $\Por$ is \emph{${<}\kappa$-$F$-linked} if it is $\mu$-$F$-linked for some infinite cardinal $\mu<\kappa$. Likewise, define \emph{${<}\kappa$-uf-linked}.
    \item The poset $\Por$ is \emph{$\mu$-$F$-Knaster} if any subset of $\Por$ of size $\mu$ contains an $F$-linked set of size $\mu$.

    Say that $\Por$ is \emph{$\mu$-uf-Knaster} if any subset of $\Por$ of size $\mu$ contains a uf-linked set of size $\mu$.\footnote{In general, this notion is stronger than ``$\Por$ is $\mu$-$D$-Knaster for every non-principal ultrafilter $D$ on $\omega$". Likewise for the notion of $\mu$-uf-linked.}
\end{enumerate}
\end{definition}

When $F\subseteq F'$ are free filters, it is clear that any $F'$-linked set is $F$-linked. In particular, a set is uf-linked iff it is $F$-linked for every free filter $F$. Though $\Fr$-linked is the weakest, and uf-linked is the strongest among these properties, they are equivalent for some posets.

\begin{lemma}[{\cite[Lemma~5.5]{mejiavert}}]\label{quasiuf}
Let $\Por$ be a poset.
  \begin{enumerate}[(a)]
     \item If $F$ is a free filter on $\omega$ generated by ${<}\pfrak$-many sets, then any subset of $\Por$ is $F$-linked iff it is $\Fr$-linked.
     \item If $\Por$ has $\pfrak$-cc then any subset of $\Por$ is uf-linked iff it is $\Fr$-linked.
  \end{enumerate}
\end{lemma}


\begin{proof}
   We repeat the proof for completeness. It is enough to show that, if $\Por$ is a poset and $F$ is a free filter on $\omega$ such that either $F$ is generated by ${<}\pfrak$-many sets or $\Por$ is $\pfrak$-cc, then any $\Fr$-linked subset of $\Por$ is $F$-linked. Towards a contradiction, assume that $Q\subseteq\Por$ is $\Fr$-linked but not $F$-linked, so there are a countable sequence $\la p_n:n<\omega\ra$ in $Q$, a maximal antichain $A\subseteq\Por$ and a sequence $\la a_r:r\in A\ra$ in $F$ such that each $r\in A$ is incompatible with $p_n$ for every $n\in a_r$. In any of the two cases of the hypothesis, it can be concluded that there is some pseudo-intersection $a\in[\omega]^{\aleph_0}$ of $\la a_r:r\in A\ra$. Hence each $r\in A$ forces $p_n\in\dot{G}$ for only finitely many $n\in a$, which means that $\Por$ forces the same. However, since $Q$ is $\Fr$-linked, there is some $q\in\Por$ that forces $\exists^\infty n\in a(p_n\in\dot{G})$, a contradiction.
\end{proof}

\begin{remark}\label{RemKnaster}
   Let $\Por$ be a poset and $\mu$ an infinite cardinal.
   \begin{enumerate}[(1)]
       \item Any $\Fr$-linked subset of $\Por$ cannot contain infinite antichains of $\Por$, that is, it is \emph{finite-cc}.\footnote{Say that $Q\subseteq\Por$ is \emph{finite-cc} if every antichain of $\Por$ contained in $Q$ is finite.}
       \item Any $\mu$-$\Fr$-linked poset is $\mu$-finite-cc (i.e., the union of ${\leq}\mu$-many finite-cc sets). As ``finite-cc" is an absolute property for transitive models of $\thzfc$,\footnote{Let $\Por$ be a poset and $Q\subseteq\Por$. Consider the tree $T_Q\subseteq Q^{<\omega}$ defined by $t\in T_Q$ iff $\{t(k):k<|t|\}$ is an antichain of $\Por$. Note that $Q$ is finite-cc iff $T_Q$ does not have an infinite branch, which is an absolute property for transitive models of $\thzfc$.} ``$\mu$-finite-cc" is also absolute.

       \item Any $\mu$-$\Fr$-linked poset is $\mu^+$-$\Fr$-Knaster.
       \item By (1), if $\theta$ is an infinite cardinal then any $\theta$-$\Fr$-Knaster poset is $\theta$-finite-cc-Knaster (that is, any subset of the poset of size $\theta$ contains a finite-cc set of size $\theta$). Also, any $\theta$-finite-cc-Knaster poset has the $\theta$-Knaster property because, by Erd\H{o}s--Dushnik--Miller~\cite{DuMi}, every finite-cc set of size $\theta$ contains a linked set of the same size.
       \item It is clear that any singleton is uf-linked. Hence, any poset of size ${\leq}\mu$ is $\mu$-uf-linked.
       \item Assume that $\Por$ is a $\pfrak$-cc poset. In view of Lemma~\ref{quasiuf}, $\Por$ is $\mu$-$\Fr$-linked iff it is $\mu$-uf-linked. In the same way, $\Por$ is $\theta$-$\Fr$-Knaster iff it is $\theta$-uf-Knaster. Note that, for $\theta\leq\pfrak$, $\theta$-$\Fr$-Knaster implies $\theta$-Knaster (and hence $\pfrak$-cc) by (4). 
   \end{enumerate}
\end{remark}

Now we look at instances of $\sigma$-uf-linked posets. The following result indicates that random forcing is $\sigma$-uf-linked.

\begin{lemma}
   Any complete Boolean algebra that admits a strictly-positive $\sigma$-additive measure is $\sigma$-uf-linked. In particular, any random algebra is $\sigma$-uf-linked.
\end{lemma}
\begin{proof}
   Any such algebra is $\sigma$-$\Fr$-linked by~\cite[Lemma~3.29]{mejiavert}, so it is $\sigma$-uf-linked by Lemma~\ref{quasiuf}.
\end{proof}



We also show that any poset of the form $\Eor^h_b$ (see Definition~\ref{DefEDforcing}) is $\sigma$-uf-linked. This actually follows the idea of Miller's proof that $\Eor$ is $\Dbf$-good (see~\cite{Mi}, in fact, his proof indicates that $\Eor$ is $\sigma$-uf-linked). To see this, we use ultrafilter limits.

\begin{definition}\label{Deffamlim}
  Let $D$ be an ultrafilter on $\Pwf(\omega)$, $X$ a topological space. If $\bar{x}=\la x_n:n<\omega\ra$  is a sequence on $X$ and $x\in X$, we say that \emph{$\bar{x}$ $D$-converges to $x$} if, for every open neighborhood $U$ of $x$, $\{n<\omega:x_n\in U\}\in D$. Here, we also say that \emph{$x$ is a $D$-limit of $\bar{x}$}.
\end{definition}

Note that there is at most one $D$-limit for Hausdorff spaces. In this case, we denote by $\lim^D_n x_n$ the ultrafilter limit of $\bar{x}$. Existence can always be guaranteed from compactness.

\begin{lemma}\label{existsfamlim}
   If $X$ is a compact Hausdorff space and $D$ is an ultrafilter on $\omega$, then any countable sequence in $X$ has a unique ultrafilter limit.
\end{lemma}
\begin{proof}
  Towards a contradiction, assume that there is a sequence $\la x_n:n<\omega\ra$ on $X$ without $D$-limit. So, for any $x\in X$, there is some open neighborhood $U_x$ of $x$ such that $a_x:=\{n<\omega : x_n\notin U_x\}\in D$. By compactness, there is some finite $F\subseteq X$ such that $\bigcup_{x\in F}U_x=X$. On the other hand, $\bigcap_{x\in F}a_x\in D$, but $\bigcap_{x\in F}a_x=\{n<\omega : x_n\notin X\}=\emptyset$, a contradiction.
\end{proof}


\begin{example}\label{Exmplim}
   Recall from Definition~\ref{Defrandom} that $\lambda^*$ denotes the Lebesgue measure on $2^\omega$.
   \begin{enumerate}[(1)]
       \item Given a real $\delta\in(0,1)$ the set $\Bor_{\geq\delta}:=\{T\in\Bor:\lambda^*([T])\geq\delta\}$ is a compact subspace of $2^{2^{<\omega}}$ (with the Cantor-space topology).
       Therefore, every sequence in $\Bor_{\geq\delta}$ has its $D$-limit in $\Bor_{\geq\delta}$ for every ultrafilter $D$ on $\Pwf(\omega)$. Even more, if $\bar{p}=\la p_n:n<\omega\ra$ is a sequence in $\Bor_{\geq\delta}$, then $q=\lim^D_n p_n$ satisfies that, for any $t\in2^{<\omega}$, $t\in q$ iff $\{n<\omega: t\in p_n\}\in D$.

       \item Fix $b,h$ as in Definition~\ref{DefEDforcing}. Note that, for any $h'\in\omega^\omega$, $\Swf(b,h')$ is a compact subspace of $\Pwf(\omega)^\omega$ (with the product topology where $\Pwf(\omega)$ is the Cantor space) so, for any $m<\omega$, $\Swf(b,m\cdot h)$ is a compact space. Therefore, if $D$ is an ultrafilter on $\omega$, $s\in\seq_{<\omega}(b)$ and $\bar{p}=\la p_n:n<\omega\ra$, $p_n=(s,\varphi_n)$, is a sequence in $\Eor^h_b(s,m)$, the sequence $\la\varphi_n:n<\omega\ra$ has its $D$-limit $\varphi$ in $\Swf(b,m\cdot h)$ for any ultrafilter $D$ on $\Pwf(\omega)$. In this case, we say that \emph{$\lim^D_n p_n:=(s,\varphi)$ is the $D$-limit of $\bar{p}$}. Note that, for any $k<\omega$, $k\in\varphi(i)$ iff $\{n<\omega: k\in\varphi_n(i)\}\in D$.
   \end{enumerate}
\end{example}

\begin{lemma}\label{EDuflim}
   Let $D$ be a non-principal ultrafilter on $\Pwf(\omega)$ and $b,h$ as in Definition~\ref{DefEDforcing}. If $G$ is $\Eor^h_b$-generic over $V$ then, in $V[G]$, $D$ can be extended to an ultrafilter $D^*$ on $\Pwf(\omega)\cap V[G]$ such that, for any $(s,m)\in\seq_{<\omega}(b)\times\omega$ and any sequence $\bar{p}\in\Eor_b^h(s,m)\cap V$ that has its $D$-limit in $G$, $\dot{W}(\bar{p})[G]\in D^*$.

   In particular, $\Eor^h_b$ is $\sigma$-uf-linked.
\end{lemma}

This lemma is a direct consequence of the following claim in $V$.
   \begin{clm}
      Assume $N<\omega$, $\{(s_k,m_k):k<N\}\subseteq\seq_{<\omega}(b)\times\omega$, $\{\bar{p}^k:k<N\}$ such that each $\bar{p}^k=\la p_{k,n}:n<\omega\ra$ is a sequence in $\Eor^h_b(s_k,m_k)$, $q_k$ is the $D$-limit of $\bar{p}^k$ for each $k<N$, and $q\in\Eor^h_b$ is stronger than every $q_k$. If $a\in D$ then $q$ forces that $a\cap\bigcap_{k<N}\dot{W}(\bar{p}^k)\neq\emptyset$.
   \end{clm}
   \begin{proof}
      We can express the forcing conditions as $p_{k,n}=(s_k,\varphi_{k,n})$, $q_k=(s_k,\varphi_k)$ where each $\varphi_k$ is the $D$-limit of $\la\varphi_{k,n}:n<\omega\ra$ in $\Swf(b,m_k\cdot h)$. Assume that $q'=(t,\psi)\leq q$ in $\Eor^h_b$. Wlog, by making $q'$ stronger, we can assume that $m^*\cdot h(i)<b(i)$ for any $i\geq|t|$ where $m^*:=m_{q'}+\sum_{k<N}m_k$. Note that $U_k:=\{\varphi\in\Swf(b,m_k\cdot h):\forall i\in|t|\menos|s_k|(t(i)\notin\varphi(i))\}$ is an open neighborhood of $\varphi_k$ in $\Swf(b,m_k\cdot h)$, so $\{n<\omega:\forall i\in|t|\menos|s_k|(t(i)\notin\varphi_{k,n}(i))\}\in D$. Hence $a\cap\bigcap_{k<N}\{n<\omega:\forall i\in|t|\menos|s_k|(t(i)\notin\varphi_{k,n}(i))\}$
      is non-empty. Choose an $n$ in that set and put $r:=(t,\psi')$ where $\psi'(i):=\psi(i)\cup\bigcup_{k<N}\varphi_{k,n}(i)$. This is a condition in $\Eor^h_b$ because $|\psi'(i)|\leq m^*\cdot h(i)$ for every $i<\omega$, and $m^*\cdot h(i)<b(i)$ for $i\geq|t|$. Moreover, $r$ is stronger than $q'$ and $p_{n,k}$ for any $k<N$, so it forces $n\in a\cap\bigcap_{k<N}\dot{W}(\bar{p}^k)$.
   \end{proof}

\begin{remark}\label{Remnorandomlim}
   By Lemma~\ref{EDuflim}, if $D$ is a non-principal ultrafilter on $\omega$ and $\bar{p}$ is a countable sequence in $\Eor^h_b(s,m)$ then its $D$-limit forces that $\dot{W}(\bar{p})$ is infinite. However, 
   this may not be true for $\Bor_{\geq\delta}$ of Example~\ref{Exmplim}(1).
   For instance, let $0<k<\omega$, $\delta:=1-2^{-k}$ and let $\la I_n:n<\omega\ra$ be an interval partition of $[k,\omega)$ such that $\sum_{n<\omega}2^{-|I_n|}<1$. For each $n<\omega$ define
   \[p_n:=\{t\in 2^{<\omega}:\text{if $|t|\geq k$ and $t(i)=0$ for all $i<k$, then $t(i)=0$ for all $i\in I_n\cap|t|$.}\}\]
   It can be shown that $p_n\in\Bor_{\geq\delta}$ and that $q:=2^{<\omega}$ is the $D$-limit of $\bar{p}:=\la p_n:n<\omega\ra$. As $\lambda^*([p_n])=1-2^{-k}+2^{-k-|I_n|}$, $\lambda^*(\bigcup_{n<\omega}[p_n])\leq1-2^{-k}+2^{-k}\sum_{n<\omega}2^{-|I_n|}<1$, so $[q]\menos\bigcup_{n<\omega}[p_n]$ has positive measure. Hence, there is an $r\in\Bor$ such that $[r]\cap\bigcup_{n<\omega}[p_n]=\emptyset$, so $r$ forces that $\dot{W}(\bar{p})=\emptyset$.
\end{remark}

\subsection{Preservation of strongly unbounded families and of mad families}

Linkedness and Knaster notions associated with filters actually work to preserve strongly $\Dbf$-unbounded families and certain type of mad families (see Definition~\ref{DefMd}). In~\cite{mejiavert} it was proved that $\mu$-$\Fr$-linked posets satisfy stronger properties than these type of preservation, for instance,

\begin{theorem}[{\cite[Thm.~3.30]{mejiavert}}]\label{ufextGood}
   Any $\mu$-$\Fr$-linked poset is $\mu^+$-$\mathbf{D}$-good. In particular, it preserves all the strongly $\kappa$-$\Dbf$-unbounded families from the ground model for any regular $\kappa\geq\mu^+$.
\end{theorem}



The preservation of strongly unbounded families via Frechet-Knaster posets actually generalizes~\cite[Main Lemma~4.6]{GMS}.

\begin{theorem}\label{FrKnasterpresunb}
   If $\kappa$ is an uncountable regular cardinal then any $\kappa$-$\Fr$-Knaster poset preserves all the strongly $\kappa$-$\Dbf$-unbounded families from the ground model.
\end{theorem}
\begin{proof}
   Let $\Por$ be a $\kappa$-$\Fr$-Knaster poset and let $F\subseteq\omega^\omega$ be a strongly $\kappa$-$\Dbf$-unbounded family in the ground model. Towards a contradiction, assume that there is a $\Por$-name $\dot{h}$ of a real in $\omega^\omega$ and a $p\in\Por$ such that $p\Vdash|\{x\in F:x\leq^*\dot{h}\}|\geq\kappa$.
   Find $F'\subseteq F$ of size $\kappa$, a family of conditions $\{p_x:x\in F'\}\subseteq\Por$ and a natural number $m$ such that, for each $x\in F'$, $p_x\leq p$ and $p_x\Vdash\forall n\geq m(x(n)\leq\dot{h}(n))$. As $\Por$ is $\kappa$-$\Fr$-Knaster, there is some $F''\subseteq F'$ of size $\kappa$ such that $\{p_x:x\in F''\}$ is $\Fr$-linked.

   Note that there is a $j\geq m$ such that the set $\{x(j):x\in F''\}$ is infinite. (otherwise $F''$ would be bounded, which contradicts that $F$ is strongly $\kappa$-$\Dbf$-unbounded). Choose $\{x_n:n<\omega\}\subseteq F''$ such that $x_n(j)\neq x_{n'}(j)$ whenever $n\neq n'$. For each $n<\omega$, put $p_n:=p_{x_n}$. As $\bar{p}=\la p_n:n<\omega\ra$ is a sequence in a $\Fr$-linked set, there is a condition $q\in\Por$ such that $q\Vdash$``$\dot{W}(\bar{p})$ is infinite". Therefore, $q$ forces that $\exists^\infty n<\omega(x_n(j)\leq\dot{h}(j))$, which is a contradiction.
\end{proof}

We now turn to preservation of mad families. The relational system defined below is inspired by~\cite{BrF}.

\begin{definition}\label{DefMd}
  Fix $A\subseteq[\omega]^{\aleph_0}$.
  \begin{enumerate}[(1)]
     \item Let $P\subseteq\big[[\omega]^{\aleph_0}\big]^{<\aleph_0}$. For $x\subseteq\omega$ and $h:\omega\times P\to\omega$, define $x\sqsubset^* h$ by
         \[\forall^\infty n<\omega\forall F\in P([n,h(n,F))\menos\bigcup F\nsubseteq x).\]
     \item Define the relational system $\Md(A):=\la[\omega]^{\aleph_0},\omega^{\omega\times[A]^{<\aleph_0}},\sqsubset^*\ra$.
     \item If $\kappa$ is an infinite cardinal, say that $A$ is a \emph{$\kappa$-strong-$\Md$ family} if $A$ is strongly $\kappa$-$\Md(A)$-unbounded. When $\kappa=\aleph_1$ we just say \emph{strong-$\Md$ family}.
  \end{enumerate}
\end{definition}

Denote $\Iwf(A):=\{w\subseteq\omega:\exists F\in[A]^{<\aleph_0}(w\subseteq^*\bigcup F)\}$. For $y\in[\omega]^{\aleph_0}\menos\Iwf(A)$ we can define a function $h_y:\omega\times[A]^{<\aleph_0}\to\omega$ such that, for every $n<\omega$ and $F\subseteq A$ finite, $y\cap [n,h_y(n,F))\menos\bigcup F\neq\emptyset$. Hence, if $x\in[\omega]^{\aleph_0}$ and $x\not\sqsubset^* h_y$ then $x\cap y$ is infinite. This actually proves the following result.

\begin{lemma}[{\cite[Lemma~3]{BrF}}]\label{DiagMain}
   Let $M$ be a transitive model of $\thzfc$ with $A\in M$. If $a^*\in[\omega]^{\aleph_0}$ is $\Md(A)$-unbounded over $M$ then $|a^*\cap y|=\aleph_0$ for any $y\in [\omega]^{\aleph_0}\cap M\menos\Iwf(A)$.
\end{lemma}

\begin{lemma}\label{HechlerMad}
   Let $Z$ be a set, $z^*\in Z$ and let $\dot{A}:=\la\dot{a}_z:z\in Z\ra$ be the a.d.\ family added by $\Hor_Z$.
   \begin{enumerate}[(a)]
       \item {\cite[Lemma~4]{BrF}} $\Hor_Z$ forces that $\dot{a}_{z^*}$ is $\Md(\dot{A}\frestr(Z\menos\{z^*\}))$-unbounded over $V^{\Hor_{Z\menos\{z^*\}}}$.

       \item If $Z$ is uncountable then $\Hor_Z$ forces that $\dot{A}$ is a strong-$\Md$ a.d.\ family.
   \end{enumerate}
\end{lemma}
\begin{proof}
   We show (b). Let $\dot{h}$ be a $\Hor_Z$-name of a function in $\omega^{\omega\times[\dot{A}]^{<\aleph_0}}$. Note that the set
   \[\{C\in[Z]^{\aleph_0}:\dot{h}\frestr(\omega\times[\dot{A}\frestr C]^{<\aleph_0})\text{\ is an $\Hor_C$-name}\}\]
   is a club in $[Z]^{\aleph_0}$ (here, $\dot{A}\frestr C:=\{\dot{a}_z:z\in C\}$), so choose some $C$ in this club set. Hence, by (a), for any $z^*\in Z\menos C$, $\Hor_{Z}$ forces that $\dot{a}_{z^*}\not\sqsubset^*\dot{h}\frestr(\omega\times[\dot{A}\frestr C]^{<\aleph_0})$, which implies that $\dot{a}_{z^*}\not\sqsubset^*\dot{h}$.
\end{proof}

\begin{theorem}\label{FrKnastermad}
   If $\kappa$ is an uncountable regular cardinal then any $\kappa$-$\Fr$-Knaster poset preserves all the $\kappa$-strong-$\Md$ families from the ground model.
\end{theorem}
\begin{proof}
   Let $\Por$ be a $\kappa$-$\Fr$-Knaster poset and let $A$ be a $\kappa$-strong-$\Md$ family. Assume, towards a contradiction, that there is some $p\in\Por$ and some $\Por$-name $\dot{h}$ of a function in $\omega^{\omega\times[A]^{<\aleph_0}}$ such that $p\Vdash|\{a\in A:a\sqsubset^*\dot{h}\}|\geq\kappa$. As in the proof of Theorem~\ref{FrKnasterpresunb}, find an $A'\subseteq A$ of size $\kappa$, $\{p_a:a\in A'\}\subseteq\Por$ and an $m<\omega$ such that, for each $a\in A'$, $p_a\leq p$ and $p_a\Vdash\forall n\geq m\forall F\in[A]^{<\aleph_0}([n,\dot{h}(n,F))\menos\bigcup F\nsubseteq a)$. We can also find an $A''\subseteq A'$ of size $\kappa$ such that $\{p_a:a\in A''\}$ is $\Fr$-linked.

   \begin{clm}\label{clFrKnastermad}
      The set of $k<\omega$ that satisfies $\exists F\in[A]^{<\aleph_0}\forall l\geq k\exists a\in A''([k,l)\menos\bigcup F\subseteq a)$ is infinite.
   \end{clm}
   \begin{proof}
      Assume the contrary, that is, there is some $k_0<\omega$ such that, for every $k\geq k_0$ and $F\in[A]^{<\aleph_0}$ there is a $g(k,F)<\omega$ such that $[k,g(k,F))\menos\bigcup F\nsubseteq a$ for all $a\in A''$. This defines a function $g\in\omega^{\omega\times[A]^{<\aleph_0}}$ that $\sqsubset^*$-dominates all the members of $A''$, but this contradicts that $A$ is strongly $\kappa$-$\Md(A)$-unbounded. This ends the proof of Claim~\ref{clFrKnastermad}.
   \end{proof}
   We continue the proof of Theorem~\ref{FrKnastermad}. Choose a $k\geq m$ and one $F\in[A]^{<\aleph_0}$ as in Claim~\ref{clFrKnastermad}. Hence, for each $l\geq k$ there is some $a_l\in A''$ such that $[k,l)\menos\bigcup F\subseteq a_l$. Put $p_l:=p_{a_l}$ and $\bar{p}:=\la p_l:l\geq k\ra$, so there is a $q\in\Por$ forcing that $\dot{W}(\bar{p})$ is infinite. Let $G$ be $\Por$-generic over $V$ with $q\in G$ and work in $V[G]$. Denote $h:=\dot{h}[G]$ and $W:=\dot{W}(\bar{p})[G]$. Note that $[k,h(k,F))\menos\bigcup F\nsubseteq a_l$ for any $l\in W$. On the other hand, $[k,l)\menos\bigcup F\subseteq a_l$ for any $l\geq k$, in particular, if $l\in W$ is chosen above $h(k,F)$ then $[k,h(k,F))\menos\bigcup F\subseteq a_l$, a contradiction. This ends the proof of Theorem~\ref{FrKnastermad}.
\end{proof}

\subsection{FS iterations and products}

To finish this section, we present some results about FS iterations and FS products of filter-linked and filter-Knaster posets. With the exception of the proof of Theorem~\ref{prodDlink}, this part was taken care of, with a more general notation, in~\cite[Sect.~5]{mejiavert}.

\begin{theorem}\label{Flinkit}
    Let $\theta$ be an uncountable regular cardinal.
    \begin{enumerate}[(a)]
        \item Any FS iteration of $\theta$-$\Fr$-Knaster posets is $\theta$-$\Fr$-Knaster.
        \item Any FS iteration of $\theta$-uf-Knaster posets is $\theta$-uf-Knaster.
        \item If $\mu$ is an infinite cardinal, then any FS iteration of length ${<}(2^\mu)^+$ of $\mu$-$\Fr$-linked posets is $\mu$-$\Fr$-linked.
    \end{enumerate}
\end{theorem}
\begin{proof}
   See~\cite[Rem.~5.11]{mejiavert}.
\end{proof}

\begin{theorem}\label{prodDlink}
   Let $\Qor_0$ and $\Qor_1$ be posets. If $D_0$ is a non-principal ultrafilter on $\omega$ and $Q_0\subseteq\Qor_0$ and $Q_1\subseteq\Qor_1$ are $D_0$-linked subsets, then $Q_0\times Q_1$ is $D_0$-linked in $\Qor_0\times\Qor_1$. In particular,
   \begin{enumerate}[(a)]
       \item The product of two $\mu$-$D_0$-linked posets is $\mu$-$D_0$-linked.
       \item If $\theta$ is regular, then the product of two $\theta$-$D_0$-Knaster posets is $\theta$-$D_0$-Knaster.
   \end{enumerate}
    Similar statements hold for ``uf-linked'' and ``uf-Knaster''.
\end{theorem}

This theorem is proved using
the following result, which is a weaker version of~\cite[Claim~1.6]{ShCov}.

\begin{lemma}\label{simpleufextension}
   Let $M\subseteq N$ be transitive models of $\thzfc$. In $M$, assume that $\Por$ is a poset, $D_0$ is an ultrafilter on $\omega$ and, in $N$, assume that $D$ is an ultrafilter that extends $D_0$. If $G$ is $\Por$-generic over $N$ and $D'_0\in M[G]$ is an ultrafilter on $\Pwf(\omega)\cap M[G]$ that extends $D_0$ then, in $N[G]$, $D\cup D'_0$ can be extended to an ultrafilter on $\Pwf(\omega)\cap N[G]$.
\end{lemma}
\begin{proof}
   Let $\dot{D}'_0\in M$ be a $\Por$-name of $D'_0$. Assume that $a\in D$, $\dot{b}\in M$ is a $\Por$-name of a member of $\dot{D}'_0$, and $p\in\Por$. Put $b'_0:=\{n<\omega:p\Vdash n\notin\dot{b}\}$. It is clear that $b'_0\in M$ and that $p\Vdash b'_0\cap\dot{b}=\emptyset$. Hence, $p\Vdash\omega\smallsetminus b'_0\in\dot{D}'_0$, which implies that $\omega\smallsetminus b'_0\in D_0$. Since $D_0\subseteq D$ and $a\in D$, $a\smallsetminus b'_0\in D$, so there is an $n\in a\smallsetminus b'_0$. Thus, in $M$, there is a $q\leq p$ that forces $n\in \dot{b}$, so $q$ forces, in $N$, that $n\in a\cap\dot{b}$.\footnote{Recall that $\Por\subseteq M$ since $\Por\in M$ and $M$ is transitive.}
\end{proof}

\begin{proof}[Proof of Theorem~\ref{prodDlink}]
    Let $\bar{q}=\la(q_{0,n},q_{1,n}):n<\omega\ra$ be a sequence in $Q_0\times Q_1$. Since both $Q_0$ and $Q_1$ are $D_0$-linked, for each $e\in\{0,1\}$ there is some $r_e\in\Qor_e$ forcing $\dot{W}_{\Qor_e}(\bar{q}_e)\in D_0^+$. Now assume that $G_0$ is $\Qor_0$-generic over $V$ and $G_1$ is $\Qor_1$-generic over $V[G_0]$ such that $(r_0,r_1)\in G_0\times G_1$. Let $M:=V$ and $N:=V[G_0]$. In $N$, there is an ultrafilter $D\supseteq D_0\cup\{W_{\Qor_0}(\bar{p})\}$ and, in $M[G_1]$, there is an ultrafilter $D'_0\supseteq D_0\cup\{W_{\Qor_1}(\bar{q})\}$. Thus, in $N[G_1]$, $D\cup D'_0$ has the finite intersection property, so $W_{\Qor_0\times\Qor_1}(\bar{q})=W_{\Qor_0}(\bar{q}_0)\cap W_{\Qor_1}(\bar{q}_1)\in D_0^+$.
\end{proof}

\begin{theorem}\label{FSprod}
   If $\kappa$ is an uncountable regular cardinal, $F$ is a free filter on $\omega$, and $\Por$ is a FS product of posets such that any finite subproduct is $\kappa$-$F$-Knaster, then $\Por$ is $\kappa$-$F$-Knaster.\footnote{In the terminology of~\cite[Sect.~5]{mejiavert}, the notion ``$F$-linked'' is \emph{FS-productive}.} In particular, when $F$ is an ultrafilter, any FS product of $\kappa$-$F$-Knaster posets is $\kappa$-$F$-Knaster (likewise for ``uf-Knaster'').
\end{theorem}

\begin{proof}
   Let $\lambda$ be a cardinal and assume that $\Por$ is the FS product of $\la\Qor_\alpha:\alpha<\lambda\ra$ as in the hypothesis. If $\la p_\zeta:\zeta<\kappa\ra\subseteq\Por$ then, by the $\Delta$-system Lemma, there is some $K\subseteq\kappa$ of size $\kappa$ such that $\la\dom p_\zeta:\zeta\in K\ra$ forms a $\Delta$-system with root $R^*$. Since $\prod_{\alpha\in R^*}\Qor_\alpha$ is $\kappa$-$F$-Knaster, we can find a $K'\subseteq K$ of size $\kappa$ such that  $\{p_\zeta\frestr R^*:\zeta\in K'\}$ is $F$-linked.

   Assume that $\la\zeta_n:n<\omega\ra\subseteq K'$. Hence, there is some $q\in\prod_{\alpha\in R^*}\Qor_\alpha$ that forces $\{n<\omega:p_{\zeta_n}\frestr R^*\in\dot{G}\}\in F^+$. As a matter of fact, $q$ forces that $\{n<\omega:p_{\zeta_n}\in\dot{G}\}\in F^+$. To see this, assume that $a\in F$ and $r\leq q$ in $\Por$. Note that $\forall^\infty n<\omega(\dom r\cap\dom p_{\zeta_n}=R^*)$. On the other hand, we can find some $s\leq r\frestr R^*$ in $\prod_{\alpha\in R^*}\Qor_\alpha$ and an $n\in a$ such that $s\leq p_{\zeta_n}\frestr R^*$ and $\dom r\cap\dom p_{\zeta_n}=R^*$. Thus
   \[r':=s\cup r\frestr(\dom r\smallsetminus R^*)\cup p_{\zeta_n}\frestr(\dom p_{\zeta_n}\smallsetminus R^*)\]
   is a condition in $\Por$ stronger than both $r$ and $p_{\zeta_n}$.

   The latter statement is a consequence of Theorem~\ref{prodDlink}.
\end{proof}

\begin{theorem}\label{prodFrlink}
   Let $\mu$ be an infinite cardinal, $\la\Qor_i : i\in I\ra$ a sequence of $\mu$-$\Fr$-linked posets witnessed by $\la Q_i(\zeta):\zeta<\mu\ra$ for each $i\in I$, and let $\Por$ be the FS product of $\la\Qor_i : i\in I\ra$. If
   \begin{enumerate}[(i)]
       \item $|I|\leq 2^\mu$ and
       \item $\prod_{i\in u}Q_i(s(i))$ is $\Fr$-linked in $\prod_{i\in u}\Qor_i$ for any finite $u\subseteq I$ and $s:u\to \mu$,
   \end{enumerate}
   then $\Por$ is $\mu$-$\Fr$-linked.\footnote{In the terminology of~\cite[Sect.~5]{mejia-temp}, if the notion ``$\Fr$-linked" is productive, then it is strongly productive.}
\end{theorem}
\begin{proof}
   By a result of Engelking and Kar{\l}owicz~\cite{TopThm}, there is a set $H\subseteq \mu^I$ of size $\leq\mu$ such that any finite partial function from $I$ to $\mu$ is extended by some function in $H$.

   For each $h\in H$ and $n<\omega$ define \[Q_{h,n}:=\{p\in\Por : |\dom p|\leq n\text{\ and }\forall i\in\dom p(p(i)\in Q_i(h(i)))\}.\]
   It is clear that these sets cover $\Por$, so it remains to show that each $Q_{h,n}$ is $\Fr$-linked. Let $\bar{p}=\la p_k:k<\omega\ra$ be a sequence in $Q_{h,n}$. By the $\Delta$-system lemma, we can find $w\subseteq\omega$ infinite such that $\la\dom p_k : k\in w\ra$ form a $\Delta$-system with root $R^*$. Hence, by (ii), there is some $q\in\prod_{i\in R^*}\Qor_i$ forcing that $\{k\in w: p_k{\upharpoonright}R^*\in\dot{G}\}$ is infinite. Similar to the last part of the proof of Theorem~\ref{FSprod}, it can be shown that $q$ forces $w\cap \dot{W}_{\Por}(\bar{p})$ is infinite.
\end{proof}

\begin{remark}
   The reason the latter proof cannot guarantee the analog result for ``$F$-linked'' for other filters $F$ in general is that, when finding the $\Delta$-system, it cannot be guaranteed that $w\in F$. However, this can be done when $F$ is a Ramsey ultrafilter, so Theorem~\ref{prodFrlink} is valid for Ramsey ultrafilters in the place of $\Fr$ (even more, (ii) is  redundant by Theorem~\ref{prodDlink}).
\end{remark}

\section{Ultrafilter-extendable matrix iterations}\label{SecConstr}

This section is dedicated to prove the main result of this work (Theorem~\ref{mainpres}).

\begin{definition}\label{Defufextmatrix}
   Let $\kappa$ be an uncountable cardinal. A \emph{${<}\kappa$-ultrafilter-extendable matrix iteration} (abbreviated \emph{${<}\kappa$-uf-extendable}) is a simple matrix iteration $\mbf$ such that, for each $\xi<\pi^\mbf$, $\Por^\mbf_{\Delta^{\mbf}(\xi),\xi}$ forces that $\Qnm^\mbf_\xi$ is a ${<}\kappa$-uf-linked poset.

   As in Definition~\ref{Defcohsys}, we omit the upper index $\mbf$ when understood.
\end{definition}

When $I^\mbf=\nu+1$ for some ordinal $\nu$, the FS iteration $\Por_{\nu,\pi}=\la\Por_{\nu,\xi},\Qnm_{\nu,\xi}:\xi<\pi\ra$ is \underline{not} a FS iteration of ${<}\kappa$-uf-linked posets in general.

\begin{definition}\label{Remufmatrixccc}
   Let $\kappa$ be uncountable regular. Given a ${<}\kappa$-uf-extendable matrix iteration $\mbf$, we define $\theta^\mbf_\xi$ and $\la \dot{Q}^\mbf_\xi(\zeta):\zeta<\theta^\mbf_\xi\ra$ for $\xi<\pi^\mbf$ as follows. By Remark~\ref{RemKnaster} (items~(2)--(4)) it can be proved by induction on $\xi\leq\pi$ that $\Por_{\alpha,\xi}$ has the $\kappa$-Knaster property for every $\alpha\in I^\mbf$. Therefore, for each $\xi<\pi^\mbf$, we can find a cardinal $\theta_\xi^\mbf<\kappa$ (in the ground model) and a sequence $\la \dot{Q}^\mbf_\xi(\zeta):\zeta<\theta^\mbf_\xi\ra$ of $\Por^\mbf_{\Delta^{\mbf}(\xi),\xi}$-names such that $\Por^\mbf_{\Delta^{\mbf}(\xi),\xi}$ forces that $\la \dot{Q}^\mbf_\xi(\zeta):\zeta<\theta^\mbf_\xi\ra$ witnesses that $\Qnm^\mbf_\xi$ is ${<}\kappa$-uf-linked. 
   Again, upper indexes are omitted when understood.
\end{definition}

\begin{theorem}\label{mainpres}
   Let $\kappa$ be an uncountable regular cardinal and $\mbf$ a ${<}\kappa$-uf-extendable matrix iteration. Then $\Por_{\alpha,\pi}$ is $\kappa$-uf-Knaster for any $\alpha\in I^\mbf$. In particular, it preserves any strongly $\kappa$-$\Dbf$-unbounded family and any $\kappa$-strong-$\Md$ family from the ground model.
\end{theorem}

Throughout this section, wlog we may assume that $I^\mbf=\gamma^\mbf$ is an ordinal (and again, we may omit the upper index).

The following is a version of the preceding result where the preserved strongly unbounded family is constructed within the matrix.

\begin{theorem}\label{mainprescohen}
   Let $\kappa\leq\mu$ be uncountable regular cardinals and let $\mbf$ be a ${<}\kappa$-uf-extendable matrix iteration. Assume that
   \begin{enumerate}[(i)]
       \item $\gamma^\mbf>\mu$ and $\pi^\mbf\geq\mu$,
       \item for each $\alpha<\mu$, $\Delta^\mbf(\alpha)=\alpha+1$ and $\Qnm^\mbf_{\alpha}=\Cor$, and
       \item $\dot{c}_\alpha$ is the $\Por_{\alpha+1,\alpha+1}$-name of the Cohen real added by $\Qnm^\mbf_{\alpha}$.
   \end{enumerate}
   Then, for any $\nu\in[\mu,\gamma^\mbf)$, $\Por_{\nu,\pi}$ forces that $\{\dot{c}_{\alpha}:\alpha<\mu\}$ forms a $\mu$-$\mathbf{D}$-strongly unbounded family.
\end{theorem}

For the proof of both results, we need to work with special conditions of the matrix and with $\Delta$-systems.

\begin{definition}\label{Defunifdeltasys}
   Let $\kappa$ be a regular uncountable cardinal and let $\mbf$ be a ${<}\kappa$-uf-extendable matrix iteration. Let $\beta<\gamma$ and $\eta\leq\pi$.
   \begin{itemize}
       \item[(1)] Define $\Por^+_{\beta,\eta}=\Por^{+\mbf}_{\beta,\eta}$ as the set of conditions $p\in\Por_{\beta,\eta}$ such that,
       for each $\xi\in\dom p$ with $\Delta(\xi)\leq\beta$, $p(\xi)$ is a $\Por_{\Delta(\xi),\xi}$-name.

       Define $\Por^*_{\beta,\eta}=\Por^{*\mbf}_{\beta,\eta}$ as the set of conditions $p\in\Por^+_{\beta,\eta}$ such that,
       for each $\xi\in\dom p$ with $\Delta(\xi)\leq\beta$, there is a $\zeta=\zeta_{p(\xi)}<\theta_\xi$ such that $\Por_{\Delta(\xi),\xi}$ forces that $p(\xi)\in\dot{Q}_{\xi}(\zeta)$.

       Note that $\Por^+_{\beta,\eta}$ is a dense subset of $\Por_{\beta,\eta}$, and $\Por^*_{\beta,\eta}$ is a dense subset of $\Por^+_{\beta,\eta}$.

       \item[(2)] For each $p\in\Por^+_{\beta,\eta}$, $\alpha\leq\beta$ and $\xi\leq\eta$, $p\frestr(\alpha,\xi)$ is the condition in $\Por^+_{\alpha,\xi}$ defined by
       \begin{itemize}
           \item[(i)] $\dom(p\frestr(\alpha,\xi))=\dom p\cap\xi$, and
           \item[(ii)] for each $\xi'\in\dom(p\frestr(\alpha,\xi))$, \[p\frestr(\alpha,\xi)(\xi')=\left\{\begin{array}{ll}
                p(\xi') &  \text{if $\Delta(\xi')\leq\alpha$,}\\
                \mathds{1} & \text{otherwise.}
           \end{array}\right.\]
       \end{itemize}
       Note that $p\frestr(\alpha,\xi)\in\Por^*_{\alpha,\xi}$ whenever $p\in\Por^*_{\beta,\eta}$

       \item[(3)] A \emph{uniform $\Delta$-system in $\Por^*_{\beta,\eta}$} is a sequence $\bar{p}=\la p_i:i\in J\ra$ of conditions in $\Por^*_{\beta,\eta}$ such that
       \begin{itemize}
           \item[(i)] $\la \dom p_i:i\in J\ra$ forms a $\Delta$-system with root $R^*$, and
           \item[(ii)] for each $\xi\in R^*$ there is a $\zeta^*_\xi<\theta_\xi$ such that $\Por_{\Delta(\xi),\xi}$ forces that $p_i(\xi)\in\dot{Q}_{\xi}(\zeta^*_\xi)$ for all $i\in J$.
       \end{itemize}


    \end{itemize}
    Note that $\Por^+_{\beta,\eta}$ and $p\frestr(\beta,\eta)$ can be defined for simple matrix iterations.
\end{definition}

The core of our main result is the following lemma.

\begin{mainlemma}\label{main}
   Let $\mbf$ be a ${<}\kappa$-uf-extendable matrix iteration with sequences of names as in Definition~\ref{Remufmatrixccc} (without assuming that $\kappa$ is regular). If $\nu\in I^\mbf$ and $\bar{p}=\la p_n:n<\omega\ra$ is a uniform $\Delta$-system in $\Por^*_{\nu,\pi}$ then there is a $q\in\Por_{\nu,\pi}$ forcing that $\dot{W}_{\Por_{\nu,\pi}}(\bar{p})$ is infinite. Moreover, if $D$ is a non-principal ultrafilter
on $\omega$ in the ground model then there is some $q\in\Por^+_{\nu,\pi}$ that forces $\dot{W}_{\Por_{\nu,\pi}}(\bar{p})\in D^+$.
\end{mainlemma}

\begin{proof}[Proof of Theorem~\ref{mainpres}]
   Let $\la p_\zeta:\zeta<\kappa\ra$ be a sequence of conditions in $\Por_{\alpha,\pi}$. For each $\zeta<\kappa$ find a $p'_\zeta\in\Por^*_{\alpha,\pi}$ stronger than $p_\zeta$. By the $\Delta$-system lemma and some easy combinatorial arguments, we can find a $K\subseteq\kappa$ of size $\kappa$ such that $\{p'_\zeta:\zeta\in K\}$ forms a uniform $\Delta$-system in $\Por^*_{\alpha,\pi}$. Therefore, by Main Lemma~\ref{main}, $\{p'_\zeta:\zeta\in K\}$ is uf-linked. Hence, $\{p_\zeta:\zeta\in K\}$ is uf-linked.
\end{proof}

\begin{proof}[Proof of Theorem~\ref{mainprescohen}]
   First note that $\Por_{\nu,\mu}$ results from a FS iteration of length $\mu$ of Cohen forcing and that, for each $\alpha<\mu$, $\dot{c}_{\alpha}$ is forced to be a Cohen real over $V_{\nu,\alpha}$. Even more, we can assume that $\theta_\alpha=\aleph_0$, $\dot{Q}_{\alpha}(n)$ is a singleton (in the ground model, not just a name), and $\omega^{<\omega}=\bigcup_{n<\omega}\dot{Q}_{\alpha}(n)$. Hence, $\Por^*_{\nu,\mu}=\Por^*_{\mu,\mu}=\Cor_\mu$.

   Towards a contradiction, assume that there is a $\Por_{\nu,\pi}$-name $\dot{h}$ of a real in $\omega^\omega$ and a $p\in\Por_{\nu,\pi}$ such that $p\Vdash_{\nu,\pi}|\{\alpha<\mu:\dot{c}_\alpha\leq^*\dot{h}\}|\geq\mu$. Find $K\subseteq\mu$ of size $\mu$, a family of conditions $\{p_\alpha:\alpha\in K\}\subseteq\Por^*_{\nu,\pi}$ and a natural number $m$ such that, for each $\alpha\in K$, $\alpha\in\dom p_\alpha$, $p_\alpha\leq p$ and $p_\alpha\Vdash\forall n\geq m(\dot{c}_\alpha(n)\leq\dot{h}(n))$. Wlog, also assume that $|p(\alpha)|\geq m$ for all $\alpha\in K$. By the $\Delta$-system lemma and some easy combinatorial arguments, we can find $K'\subseteq K$ of size $\mu$ such that $\{p_\alpha:\alpha\in K'\}$ forms a uniform $\Delta$-system in $\Por^*_{\nu,\pi}$ and there is some $t\in\omega^{<\omega}$ of length $m'\geq m$ such that, for all $\alpha\in K'$, $p_\alpha(\alpha)=t$. Choose $\{\alpha_n:n<\omega\}\subseteq K'$ (one-to-one enumeration).
   Define $p'_n$ identical to $p_{\alpha_n}$ with the sole difference that $p'_n(\alpha):=p_{\alpha_n}(\alpha)\cup\{(m',n)\}$. Note that $\bar{p}'=\la p'_n:n<\omega\ra$ forms a countable uniform $\Delta$-system. Therefore, by Main Lemma~\ref{main}, there is a condition $q\in\Por_{\nu,\pi}$ such that $q\Vdash$``$\dot{W}(\bar{p}')$ is infinite", so $q$ forces that $\exists^\infty n<\omega(\dot{c}_{\alpha_n}(m')=n\leq\dot{h}(m'))$, which is a contradiction.
\end{proof}

We now focus on the proof of Main Lemma~\ref{main}. We start with some preliminary results before developing the proof.

\begin{lemma}\label{projprop}
   Let $\mbf$ be a simple matrix iteration, $\alpha\leq\beta<\gamma$ and $\xi\leq\eta\leq\pi$. Then:
   \begin{enumerate}[(a)]
       \item For any $p\in\Por^+_{\beta,\eta}$, if $q\leq p\frestr(\alpha,\xi)$ in $\Por^+_{\alpha,\xi}$, then there is some $p'\leq p$ in $\Por^+_{\beta,\eta}$ such that $q=p'\frestr(\alpha,\xi)$.
       \item If $\beta$ is limit and $\beta\notin\ran\Delta^\mbf$ then $\Por^+_{\beta,\xi}=\limdir_{\alpha<\beta}\Por^+_{\alpha,\xi}$.
   \end{enumerate}
   Even more, similar statements hold for $\Por^*_{\beta,\eta}$ when $\mbf$ is a ${<}\kappa$-uf-extendable matrix iteration.
\end{lemma}
\begin{proof}
   To see (a), define $p'$ such that $\dom p'=\dom p\cup\dom q$ and $p'(\xi)$ is determined by the following cases: when $\xi\in\dom p\smallsetminus\dom q$, $p'(\xi):=p(\xi)$; when $\xi\in\dom q$, put $p'(\xi):=p(\xi)$ if $\alpha<\Delta(\xi)$, otherwise $p'(\xi):=q(\xi)$.

   Now we show (b) by induction on $\xi$. The case $\xi=0$ and the limit step are immediate. For the successor step, assume that $\Por^+_{\beta,\xi}=\limdir_{\alpha<\beta}\Por^+_{\alpha,\xi}$. If $\beta<\Delta(\xi)$ then $\Por^+_{\alpha,\xi+1}=\Por^+_{\alpha,\xi}\ast\mathds{1}$ for any $\alpha\leq\beta$, so the conclusion follows; if $\Delta(\xi)\leq\beta$ then $\Delta(\xi)<\beta$ (because $\Delta(\xi)\neq\beta$) and, whenever $p\in\Por^+_{\beta,\xi+1}$, by induction hypothesis $p\frestr(\beta,\xi)\in\Por^+_{\alpha,\xi}$ for some $\alpha\in[\Delta(\xi),\beta)$. On the other hand, $p(\xi)$ is a $\Por_{\Delta(\xi),\xi}$-name of a condition in $\Qnm_\xi$, so $p\in\Por^+_{\alpha,\xi+1}$.
\end{proof}

\begin{lemma}\label{linearufextension}
   Let $\Por_\pi=\la\Por_\xi,\Qnm_\xi:\xi<\pi\ra$ be a FS iteration with $\pi$ limit. Assume:
   \begin{enumerate}[(i)]
       \item $\bar{p}=\la p_n:n<\omega\ra$ is a sequence of conditions in $\Por_\pi$.
       \item $\la\dot{D}_{\xi}:\xi<\pi\ra$ is a sequence such that each $\dot{D}_{\xi}$ is a $\Por_{\xi}$-name of a non-principal ultrafilter on $\omega$ that contains $\dot{D}_{\xi_0}$ for any $\xi_0<\xi$.
       \item $q\in\Por_\pi$.
       \item For any $\xi<\pi$, $q\frestr\xi$ forces that $\dot{W}_{\Por_\xi}(\bar{p}\frestr\xi)\in\dot{D}_{\xi}$.
   \end{enumerate}
   Then $q$ forces that $\bigcup_{\xi<\pi}\dot{D}_{\xi}\cup\{\dot{W}_{\Por_\pi}(\bar{p})\}$ can be extended to an ultrafilter.
\end{lemma}
\begin{proof}
   Let $r\leq q$ in $\Por_\pi$ and $\dot{b}$ a $\Por_\pi$-name of a member of $\bigcup_{\xi<\pi}\dot{D}_{\xi}$. Wlog (by strengthening $r$ if necessary), we may assume that there is a $\xi<\pi$ such that $r,q\in\Por_\xi$ and $\dot{b}$ is (forced to be equal to) a $\Por_\xi$-name of a member of $\dot{D}_\xi$. By (iv), there are some $r'\leq r$ in $\Por_\xi$ and an $n<\omega$ such that $r'\leq p_n\frestr\xi$ and $r'\Vdash_\xi n\in\dot{b}$. Hence, $q':=r'\cup p_n\frestr[\xi,\pi)$ forces in $\Por_\pi$ that $n\in\dot{b}\cap\dot{W}_{\Por_\pi}(\bar{p})$.
\end{proof}

\begin{lemma}\label{limufextension}
   Let $\sbf$ be a simple matrix iteration with $I^\sbf=\{0,1\}$. Assume:
   \begin{enumerate}[(i)]
       \item $\pi=\pi^\sbf$ is limit.
       \item $\bar{p}=\la p_n:n<\omega\ra$ is a sequence of conditions in $\Por^+_{1,\pi}$.
       \item $\la\dot{D}_{i,\xi}:i<2,\xi<\pi\ra$ is a sequence such that each $\dot{D}_{i,\xi}$ is a $\Por_{i,\xi}$-name of a non-principal ultrafilter on $\omega$ that contains $\dot{D}_{i_0,\xi_0}$ for any $i_0\leq i$ and $\xi_0\leq\xi$.
       \item $\dot{D}_{0,\pi}$ is a $\Por_{0,\pi}$-name of an ultrafilter containing $\bigcup_{\xi<\pi}\dot{D}_{0,\xi}$.
       \item $q\in\Por^+_{1,\pi}$.
       \item For any $\xi<\pi$, $q\frestr(1,\xi)$ forces that $\dot{W}_{\Por_{1,\xi}}(\bar{p}\frestr(1,\xi))\in\dot{D}_{1,\xi}$.
       \item $q\frestr(0,\pi)$ forces that $\dot{W}_{\Por_{0,\pi}}(\bar{p}\frestr(0,\pi))\in\dot{D}_{0,\pi}$.
   \end{enumerate}
   Then, $q$ forces that $\dot{D}_{0,\pi}\cup\bigcup_{\xi<\pi}\dot{D}_{1,\xi}\cup\{\dot{W}_{\Por_{1,\pi}}(\bar{p})\}$ can be extended to an ultrafilter. Even more, $\mathds{1}_{\Por_{1,\pi}}$ forces that $\dot{D}_{0,\pi}\cup\bigcup_{\xi<\pi}\dot{D}_{1,\xi}$ can be extended to an ultrafilter.
\end{lemma}
\begin{proof}
   We show that, for any $\Por_{1,\pi}$-names $\dot{a}$ and $\dot{b}$ of members of $\bigcup_{\xi<\pi}\dot{D}_{1,\xi}$ and $\dot{D}_{0,\pi}$, respectively, $q$ forces that $\dot{a}\cap\dot{b}\cap\dot{W}_{\Por_{1,\pi}}(\bar{p})\neq\emptyset$. Let $r\leq q$ in $\Por^+_{1,\pi}$. Wlog (by strengthening $r$ if necessary) we may assume that $\dot{b}$ is a $\Por_{0,\pi}$-name and that there is a $\xi<\pi$ such that $\dot{a}$ is a $\Por_{1,\xi}$-name, $r,q\in\Por^+_{1,\xi}$ and $r$ forces that $\dot{a}\in\dot{D}_{1,\xi}$. Consider the $\Por_{0,\xi}$-name
   \[\dot{b}'_0:=\{n<\omega:p_n\frestr(0,\xi)\in\dot{G}_{0,\xi}\text{\ and }p_n\frestr(0,\pi)\Vdash_{\Por_{0,\pi}/\Por_{0,\xi}}n\notin\dot{b}\}.\]
   It is clear that $\Vdash_{0,\pi}\dot{b}\cap\dot{W}_{\Por_{0,\pi}}(\bar{p}\frestr(0,\pi))\cap\dot{b}'_0=\emptyset$ so, by (vii), $r\frestr(0,\xi)$ forces in $\Por_{0,\xi}$ that
   \[\dot{b}_0:=\dot{W}_{\Por_{0,\xi}}(\bar{p}{\upharpoonright}(0,\xi))\menos\dot{b}'_0=\{n<\omega:p_n\frestr(0,\xi)\in\dot{G}_{0,\xi}\text{\ and }p_n\frestr(0,\pi)\nVdash_{\Por_{0,\pi}/\Por_{0,\xi}}n\notin\dot{b}\}\in\dot{D}_{0,\xi}.\]
   Hence $r\Vdash_{1,\xi}\dot{b}_0\in\dot{D}_{1,\xi}$, so by (vi) $r$ forces that
   \[\dot{a}\cap\dot{b}_0\cap\dot{W}_{\Por_{1,\xi}}(\bar{p}\frestr(1,\xi))\in\dot{D}_{1,\xi}.\]
   Find $n<\omega$ and $r'\in\Por^+_{1,\xi}$ stronger than both $r$ and $p_n\frestr(1,\xi)$ such that $r'\Vdash_{1,\xi} n\in\dot{a}\cap\dot{b}_0$. This implies that $r'\frestr(0,\xi)\Vdash_{0,\xi} n\in\dot{b}_0$, so there is a condition $s\leq p_n\frestr(0,\pi)$ in $\Por_{0,\pi}^+$ such that $s\frestr(0,\xi)\leq r'\frestr(0,\xi)$ and $s\Vdash_{0,\pi} n\in\dot{b}$. Now, if we put $p':=r'\cup p_n\frestr[\xi,\pi)$, then $s\leq p'\frestr(0,\pi)$, which implies by Lemma~\ref{projprop}(a) that $s=p''\frestr(0,\pi)$ for some $p''\leq p'$ in $\Por^+_{1,\pi}$. Hence, since $p''$ is stronger than both $s$ and $p'$, $p''$ forces that $n\in\dot{a}\cap\dot{b}$ and $p_n\in\dot{G}_{1,\pi}$.

   The ``even more" statement follows by the particular case when $q$ and every $p_n$ are the trivial condition.
\end{proof}

\begin{proof}[Proof of Main Lemma~\ref{main}]
   Recall that (in this section) $I^\mbf$ is an ordinal. Fix a uniform $\Delta$-system $\bar{p}=\la p_n:n<\omega\ra$ with root $R^*$ of $\Por^*_{\nu,\pi}$ as in Definition~\ref{Defunifdeltasys}(3) and an ultrafilter $D$ on $\omega$ (in the ground model). By recursion on $\xi\leq\pi$ we construct $\Dbf_\xi:=\la\dot{D}_{\alpha,\xi}:\alpha\leq\nu\ra$ and $\la q_{\alpha,\xi}:\alpha\leq\nu\ra$ such that, for any $\alpha\leq\nu$,
   \begin{enumerate}[(a)]
       \item $\dot{D}_{\alpha,\xi}$ is a $\Por_{\alpha,\xi}$-name of a non-principal ultrafilter on $\omega$,
       \item $\Por_{\alpha,\xi}$ forces that $D\subseteq\dot{D}_{\alpha_0,\xi_0}\subseteq\dot{D}_{\alpha,\xi}$ for any $\alpha_0\leq\alpha$ and $\xi_0\leq\xi$,
       \item $q_{\alpha,\xi}\in\Por^+_{\alpha,\xi}$ with domain $R^*\cap\xi$,
       \item $q_{\alpha,\xi}\frestr(\alpha_0,\xi_0)=q_{\alpha_0,\xi_0}$ for any $\alpha_0\leq\alpha$ and $\xi_0\leq\xi$, and
       \item $q_{\alpha,\xi}\Vdash\dot{W}_{\Por_{\alpha,\xi}}(\bar{p}\frestr(\alpha,\xi))\in\dot{D}_{\alpha,\xi}$
   \end{enumerate}
   After the construction, $q:=q_{\nu,\pi}$ is the condition we are looking for.

   \emph{Step $\xi=0$.} As $\Por_{\alpha,0}$ is the trivial poset for any $\alpha\leq\nu$, $D_{\alpha,0}:=D$ and $q_{\alpha,0}:=\mathds{1}$ work.

   \emph{Successor step.} Assume we have succeeded in our construction up to step $\xi$. For $\alpha<\Delta(\xi)$ it is clear that $\Por_{\alpha,\xi+1}\simeq\Por_{\alpha,\xi}$, so $\dot{D}_{\alpha,\xi+1}$ must be $\dot{D}_{\alpha,\xi}$.
   To define $q_{\alpha,\xi+1}$ (for all $\alpha\leq\nu$) and $\dot{D}_{\Delta(\xi),\xi+1}$ we consider two cases. If $\xi\notin R^*$ put $q_{\alpha,\xi+1}=q_{\alpha,\xi}$ and $\dot{D}_{\Delta(\xi),\xi+1}$ can be any $\Por_{\Delta(\xi),\xi+1}$-name of an ultrafilter that contains $\dot{D}_{\Delta(\xi),\xi}$ (so it also contains $\dot{D}_{\alpha,\xi+1}$ for any $\alpha<\Delta(\xi)$); if $\xi\in R^*$, since $\Qnm_{\Delta(\xi),\xi}=\Qnm_\xi$ is a $\Por_{\Delta(\xi),\xi}$-name of a ${<}\kappa$-uf-linked forcing witnessed by $\la\dot{Q}_\xi(\zeta):\zeta<\theta_\xi\ra$, and $\bar{p}(\xi):=\la p_n(\xi):n<\omega\ra$ can be seen as a $\Por_{\Delta(\xi),\xi}$-name of a sequence in $\dot{Q}_{\xi}(\zeta^*_\xi)$, there is a $\Por_{\Delta(\xi),\xi}$-name $q(\xi)$ of a member of $\Qnm_\xi$ such that $\Por_{\Delta(\xi),\xi}$ forces that
   \[q(\xi)\Vdash``\dot{W}_{\Qnm_\xi}(\bar{p}(\xi))\text{\ intersects any member of $\dot{D}_{\Delta(\xi),\xi}$"}.\]
   Put $q_{\alpha,\xi+1}:=q_{\alpha,\xi}\cup\{(\xi,q(\xi))\}$ when $\Delta(\xi)\leq\alpha\leq\nu$, otherwise $q_{\alpha,\xi+1}:=q_{\alpha,\xi}\cup\{(\xi,\mathds{1})\}$, and choose $\dot{D}_{\Delta(\xi),\xi+1}$ as a $\Por_{\Delta(\xi),\xi+1}$-name of an ultrafilter that contains $\dot{D}_{\Delta(\xi),\xi}$ and such that $q_{\Delta(\xi),\xi+1}$ forces that $\dot{W}_{\Qnm_\xi}(\bar{p}(\xi))\in\dot{D}_{\Delta(\xi),\xi+1}$.

   No matter the case, for any $\alpha<\Delta(\xi)$, $\dot{D}_{\Delta(\xi),\xi+1}$ is forced to contain $\dot{D}_{\alpha,\xi+1}$ and
   \[q_{\alpha,\xi+1}\Vdash_{\alpha,\xi+1}\dot{W}_{\Por_{\alpha,\xi+1}}(\bar{p}\frestr(\alpha,\xi+1))=\dot{W}_{\Por_{\alpha,\xi}}(\bar{p}\frestr(\alpha,\xi)),\] so this condition forces that $\dot{W}_{\Por_{\alpha,\xi+1}}(\bar{p}\frestr(\alpha,\xi+1))\in\dot{D}_{\alpha,\xi+1}$.

   Now, by induction on $\alpha\in[\Delta(\xi),\nu]$, we define $\dot{D}_{\alpha,\xi+1}$ as required. We have already dealt with the case $\alpha=\Delta(\xi)$. For the successor step, assume we have defined $\dot{D}_{\alpha,\xi+1}$ accordingly. By Lemma~\ref{simpleufextension}, we can choose a $\Por_{\alpha+1,\xi+1}$-name $\dot{D}_{\alpha+1,\xi+1}$ of an ultrafilter that contains $\dot{D}_{\alpha,\xi+1}\cup\dot{D}_{\alpha+1,\xi}$.
   For the limit step, let $\alpha$ be limit and assume we have already defined $\la\dot{D}_{\alpha_0,\xi+1}:\alpha_0<\alpha\ra$. By Lemma~\ref{simpleufextension}, for any $\alpha_0<\alpha$, $\Por_{\alpha,\xi+1}$ forces that $\dot{D}_{\alpha_0,\xi+1}\cup\dot{D}_{\alpha,\xi}$ has the finite intersection property, hence $\dot{D}_{\alpha,\xi}\cup\bigcup_{\alpha_0<\alpha}\dot{D}_{\alpha_0,\xi+1}$ also has this property, i.e., it can be extended to an ultrafilter. Let $\dot{D}_{\alpha,\xi+1}$ be a $\Por_{\alpha,\xi+1}$-name of such an ultrafilter.

   It remains to show that item (e) holds for $(\alpha,\xi+1)$ when $\Delta(\xi)\leq\alpha\leq\nu$. If $\xi\in R^*$ then $q_{\alpha,\xi+1}$ forces
   $\dot{W}_{\Por_{\alpha,\xi+1}}(\bar{p}\frestr(\alpha,\xi+1))=\dot{W}_{\Por_{\alpha,\xi}}(\bar{p}\frestr(\alpha,\xi))\cap\dot{W}_{\Qnm_\xi}(\bar{p}(\xi))$;
   else, if $\xi\notin R^*$ then $q_{\alpha,\xi+1}$ forces that \[\dot{W}_{\Por_{\alpha,\xi+1}}(\bar{p}\frestr(\alpha,\xi+1))\subseteq\dot{W}_{\Por_{\alpha,\xi}}(\bar{p}\frestr(\alpha,\xi))\text{\ and } |\dot{W}_{\Por_{\alpha,\xi}}(\bar{p}\frestr(\alpha,\xi))\smallsetminus\dot{W}_{\Por_{\alpha,\xi+1}}(\bar{p}\frestr(\alpha,\xi+1))|\leq 1\] (because $\la\dom p_n:n<\omega\ra$ forms a $\Delta$-system and $\xi$ is not in its root). Hence, in any case it is clear that $q_{\alpha,\xi+1}$ forces $\dot{W}_{\Por_{\alpha,\xi+1}}(\bar{p}\frestr(\alpha,\xi+1))\in\dot{D}_{\alpha,\xi+1}$.

   \emph{Limit step.} Let $\eta\leq\pi$ be a limit ordinal and assume we have succeeded in our construction for $\xi<\eta$. For each $\alpha\leq\nu$ put $q_{\alpha,\eta}:=\bigcup_{\xi<\eta}q_{\alpha,\xi}$, which clearly satisfies (c) and (d). By recursion on $\alpha\leq\nu$ we define $\dot{D}_{\alpha,\eta}$ satisfying (a), (b) and (e). 
   When $\alpha=0$, by Lemma~\ref{linearufextension} applied to the FS iteration $\Por^+_{0,\eta}=\la\Por^+_{0,\xi},\Qnm_{0,\xi}:\xi<\eta\ra$, $q_{0,\eta}$ forces that $\dot{W}_{\Por_{0,\eta}}(\bar{p}\frestr(0,\eta))$ intersects any member of $\bigcup_{\xi<\eta}\dot{D}_{0,\xi}$, so we can find a $\Por_{0,\eta}$-name of an ultrafilter $\dot{D}_{0,\eta}$ that contains this union and such that $q_{0,\eta}$ forces $\dot{W}_{\Por_{0,\eta}}(\bar{p}\frestr(0,\eta))\in\dot{D}_{0,\eta}$.

   For the successor step, assume we have found $\dot{D}_{\alpha,\eta}$. By Lemma~\ref{limufextension} applied to $(\mbf|\{\alpha,\alpha+1\})\frestr\eta$, $q_{\alpha+1,\eta}$ forces that $\dot{D}_{\alpha,\eta}\cup\bigcup_{\xi<\eta}\dot{D}_{\alpha+1,\xi}\cup\{\dot{W}_{\Por_{\alpha+1,\eta}}(\bar{p}\frestr(\alpha+1,\eta))\}$ has the finite intersection property, so we can find a $\Por_{\alpha+1,\eta}$-name $\dot{D}_{\alpha+1,\eta}$ that satisfies (a), (b) and (e).

   For the limit step, let $\alpha\leq\nu$ limit and assume we have defined $\dot{D}_{\alpha_0,\eta}$ for all $\alpha_0<\alpha$. By Lemma~\ref{limufextension} applied to $(\mbf|\{\alpha_0,\alpha\})\frestr\eta$, $q_{\alpha,\eta}$ forces that $\dot{D}_{\alpha_0,\eta}\cup\bigcup_{\xi<\eta}\dot{D}_{\alpha,\xi}\cup\{\dot{W}_{\Por_{\alpha,\eta}}(\bar{p}\frestr(\alpha,\eta))\}$ has the finite intersection property. Hence, $q_{\alpha,\eta}$ forces that $\bigcup_{\alpha_0<\alpha}\dot{D}_{\alpha_0,\eta}\cup\bigcup_{\xi<\eta}\dot{D}_{\alpha,\xi}\cup\{\dot{W}_{\Por_{\alpha,\eta}}(\bar{p}\frestr(\alpha,\eta))\}$ has the same property, so it can be extended to an ultrafilter $\dot{D}_{\alpha,\eta}$.
\end{proof}

\section{Applications}\label{SecAppl}

In this section, we show the applications of our main result. We start with the following notion of strongly dominating family.

\begin{definition}\label{Defthetadom}
   Let $\Rbf=\la X,Y,\sqsubset\ra$ be a relational system and let $\theta$ be a cardinal number. A subset $D$ of $Y$ is a \emph{strongly $\theta$-$\Rbf$-dominating family} if there is a ${<}\theta$-directed partial order $\la L,\unlhd\ra$ such that $D=\{a_l:l\in L\}$ and, for any $x\in X$, there is some $l_0\in L$ such that $x\sqsubset a_l$ for all $l\unrhd l_0$ in $L$.
\end{definition}

This notion is used implicitly in~\cite{KTT,GKS} as a dual of strongly unbounded families. Since strongly unbounded families give upper bounds of $\bfrak(\Rbf)$ and lower bounds of $\dfrak(\Rbf)$, strongly unbounded families are used to find the converse bounds.

\begin{lemma}
   If $D\subseteq Y$ is a $\theta$-$\Rbf$-dominating family then $\theta\leq\bfrak(\Abf)$ and $\dfrak(\Abf)\leq|D|\leq|L|$.
\end{lemma}

Note that, whenever $\theta$ is regular, any strongly $\theta$-$\Rbf$-unbounded family is strongly $\theta$-$\Rbf^\perp$-dominating, where $\Rbf^\perp:=\la Y,X,\not\sqsupset\ra$ is the relational system dual to $\Rbf$. Also recall that, for any cardinal $\theta$ of uncountable cofinality, any ${<}\theta$-directed partial order is still ${<}\theta$-directed in any ccc-forcing extension.

In principle, there is no need to use strongly dominating families to prove our consistency results that do not depend on large cardinals (as in the proof of Theorem~\ref{Apl2}). However, considering the techniques from~\cite{KTT,GKS}, such families are required to separate more cardinals in the right side of Cicho\'n's diagram after applying Boolean ultrapowers.

\begin{theorem}\label{Apl1}
Let $\theta_0\leq \theta_1\leq \theta_2\leq \mu\leq \nu$ be uncountable regular cardinals and let $\lambda$ be a cardinal such that $\nu\leq\lambda=\lambda^{{<}\theta_2}$. Then there is a ccc poset that forces $\add(\Ncal)=\theta_0$, $\cov(\Ncal)=\theta_1$, $\bfrak=\afrak=\theta_2$, $\non(\Mcal)=\mu$, $\cov(\Mcal)=\nu$ and $\dfrak=\non(\Ncal)=\cfrak=\lambda$ (see Figure~\ref{Fig7val} on page~\pageref{Fig7val}).
\end{theorem}

\begin{proof} Denote $\Sor_0=\Loc$, $\Sor_{1}=\Bor$ (defined in Section~\ref{SecPre}) and $\Sor_2=\Dor$, the Hechler poset for adding a dominating real. Fix a bijection $g=(g_0,g_1,g_2):\lambda\to3\times\lambda\times\lambda$ and a function $t:\nu\mu\to\nu$ such that $t(\nu\delta+\alpha)=\alpha$ for each $\delta<\mu$ and $\alpha<\nu$. For each $\rho<\nu\mu$ denote $\eta_\rho:=\nu+\lambda\rho$, and put $R_i:=\{\eta_\rho+1+\varepsilon:\varepsilon<\lambda,\ \rho<\nu\mu,\ g_0(\varepsilon)=i\}$ for each $i<3$. Set $R:=R_0\cup R_1\cup R_2$.

The poset we want is $\Hor_{\theta_2}\ast\Cor_{\lambda}\ast\Por$ where $\Por$ is constructed in $V_{0,0}:=V^{\Hor_{\theta_2}\ast\Cor_\lambda}$ from a ${<}\theta_2$-uf-extendable matrix iteration $\mbf$, with $I^\mbf=\nu+1$ and $\pi^\mbf=\nu+\lambda\nu\mu$, such that

\begin{enumerate}[(I)]
    \item for any $\alpha<\nu$, $\Delta^\mbf(\alpha)=\alpha+1$ and $\Qnm^\mbf_{\Delta(\alpha),\alpha}=\omega^{<\omega}$,
\end{enumerate}

and the matrix iteration at each interval of the form $[\eta_\rho,\eta_{\rho+1})$ for $\rho<\nu\mu$ is defined as follows. Assume that $\mbf{\upharpoonright}\eta_\rho$ has been constructed and that, for any $i<3$ and $\xi\in R_i\cap\eta_\rho$, a $\Por_{\Delta(\xi),\xi}$-name $\dot{N}_\xi$ of a transitive model of $\thzfc$ of size ${<}\theta_i$ has already been defined.

Choose
   \begin{enumerate}
       \item[(0)] for $i\in\{0,2\}$, an enumeration $\{\dot{x}_{i,\zeta}^\rho:\zeta<\lambda\}$ of all the nice $\Por_{\nu,\eta_\rho}$-names for all the members of $\omega^\omega$; for $i=1$, $\{\dot{x}_{i,\zeta}^\rho:\zeta<\lambda\}$ enumerates all the (nice) $\Por_{\nu,\eta_\rho}$-names for all the members of $\Omega$ (from $\Cn$, see Example~\ref{ExmPrs}(3));
       \item[(1)] for $i<3$, an enumeration $[\eta_\rho\cap R_i]^{{<}\theta_i}=\{A_{i,\zeta}^\rho:\zeta<\lambda\}$.
   \end{enumerate}

For $\xi\in[\eta_\rho,\eta_{\rho+1})$,
\begin{enumerate}
    \item[(II)] if $\xi=\eta_\rho$, put $\Delta^\mbf(\xi)=t(\rho)+1$ and $\Qnm_{\xi}^{\mbf}=\Eor^{V_{\Delta(\xi),\xi}}$;
     \item[(III)] if  $\xi=\eta_\rho+1+\varepsilon$ for some $\varepsilon<\lambda$, then there is some $\alpha<\nu$ such that $\dot{x}_{g_0(\varepsilon),g_1(\varepsilon)}^\rho$ is a $\Por_{\alpha,\eta_\rho}$-name, so we can choose
     \begin{enumerate}[({III-}1)]
         \item a successor ordinal $\Delta^\mbf(\xi)$ such that $\sup_{\gamma\in A_{g_0(\varepsilon),g_2(\varepsilon)}}\Delta(\gamma)<\Delta(\xi)$ and $\alpha<\Delta(\xi)<\nu$, and
         \item a $\Por_{\Delta(\xi),\xi}$-name $\dot{N}_\xi$ of a transitive model of $\thzfc$ of size ${<}\theta_{g_0(\varepsilon)}$ such that $\Por_{\Delta(\xi),\xi}$ forces that $\bigcup_{\gamma\in A_{g_0(\varepsilon),g_2(\varepsilon)}^\rho}\dot{N}_\gamma\subseteq \dot{N}_\xi$ and $\dot{x}_{g_0(\varepsilon),g_1(\varepsilon)}^\rho\in \dot{N}_\xi$.
     \end{enumerate}
     Put $\Qnm_{\xi}^{\mbf}=\Sor_{g_0(\varepsilon)}^{\dot{N}_\xi}$.
\end{enumerate}
According to Definition~\ref{Defufextmatrix}, the above settles the construction of $\mbf$ as a ${<}\theta_2$-uf-extendable matrix iteration. Set $\Por:=\Por_{\nu,\pi}$, which is ccc.

We need to show that $\Por$ forces the statement of the theorem. Since this poset has size $\lambda$, it forces $\cfrak\leq\lambda$. On the other hand, by Theorem~\ref{mainpres}, $\Por$ is a $\theta_2$-uf-Knaster poset, so it preserves the mad family previously added by $\Hor_{\theta_2}$ and forces $\afrak\leq\theta_2$. Even more, for any regular cardinal $\kappa\in [\theta_2,\lambda]$, $\Por$ preserves the strongly $\kappa$-$\D$-unbounded family of size $\kappa$ previously added by $\Cor_\kappa$. In particular, $\Por$ forces $\bfrak\leq \theta_2$ and $\lambda\leq \dfrak$.

Observe that $\Por$ can be obtained by the FS iteration $\la\Por_{\nu,\xi},\Qnm_{\nu,\xi}:\xi<\pi\ra$ and that all its iterands are $\theta_0$-$\Lc^*$-good and $\theta_1$-$\Cn$-good. Therefore, by Theorem~\ref{sizeforbd}, $\Por$ forces  $\add(\Ncal)\leq\theta_0$, $\cov(\Ncal)\leq\theta_1$ and $\lambda\leq\non(\Ncal)$, in fact, $\Por$ adds
\begin{enumerate}[({SU}1)]
    \item a strongly $\kappa$-$\Lc^*$-unbounded family of size $\kappa$ for each regular $\kappa\in[\theta_0,\lambda]$, and
    \item a strongly $\kappa$-$\Cn$-unbounded family of size $\kappa$ for each regular $\kappa\in[\theta_1,\lambda]$.
\end{enumerate}
On the other hand, $\Por$ adds $\mu$-cofinally many Cohen reals that form a strongly $\mu$-$\Ed$-unbounded family of size $\mu$, hence $\Por$ forces $\non(\Mcal)=\bfrak(\Ed)\leq\mu$.

To see that $\Por$ forces $\theta_0\leq\add(\Ncal)$, $\theta_1\leq\cov(\Ncal)$ and $\theta_2\leq\bfrak$, we show that $\Por$ adds the corresponding strongly dominating families. In the ground model,
order $R$ by $\eta\unlhd\eta'$ iff $\eta\leq\eta'$, $\Delta(\eta)\leq\Delta(\eta')$ and $\Vdash_{\Hor_{\theta_2}\ast\Cor_\lambda\ast\Por_{\Delta(\eta'),\eta'}}\dot{N}_\eta\subseteq \dot{N}_{\eta'}$, which is a partial order, even more, $\la R_i,\unlhd\ra$ is ${<}\theta_i$-directed for any $i<3$. To see this, if $A\subseteq R_i$ has size ${<}\theta_i$ then we can find some $\rho<\nu\mu$ such that $A\subseteq\eta_\rho$, so choose some $\zeta<\lambda$ such that $A:=A_{i,\zeta}^{\rho}$. Put $\xi:=\eta_\rho+1+\varepsilon$ where $\varepsilon=g^{-1}(i,\gamma,\zeta)$ for some $\gamma$ chosen arbitrarily. Note that $\xi$ is an upper bound of the set $A$ with respect to $\unlhd$.

In $V_{0,0}$, for $\xi\in R_0$ let $\dot{\varphi}_\xi$ be the $\Por_{\Delta(\xi),\xi+1}$-name of the $\Lc^*$-dominating slalom over $\dot{N}_\xi$ added by $\Qnm^\mbf_\xi=\Loc^{N_\xi}$; for $\xi\in R_1$ let $\dot{r}_\xi$ be the $\Por_{\Delta(\xi),\xi+1}$-name of the random real over $\dot{N}_\xi$ added by $\Qnm^\mbf_\xi=\Bor^{N_\xi}$; and for $\xi\in R_2$, let $\dot{d}_\xi$ be the $\Por_{\Delta(\xi),\xi+1}$-name of the dominating real over $\dot{N}_\xi$ added by $\Qnm^\mbf_\xi=\Dor^{N_\xi}$. Define $\dot{S}:=\{\dot{\varphi}_\xi:\xi\in R_0\}$, $\dot{C}:=\{\dot{r}_\xi:\xi\in R_1\}$, and $\dot{D}:=\{\dot{d}_\xi:\xi\in R_2\}$.

We claim that $\Por$ forces that $\dot{S}$ is a strongly $\theta_0$-$\Lc^*$-dominating family, $\dot{C}$ is a strongly $\theta_1$-$\Cn$-dominating family, and $\dot{D}$ is a strongly $\theta_2$-$\D$-dominating family. We just show this fact for $\dot{S}$ (the others can be proved similarly). Let $\dot{x}$ be a $\Por$-name for a real in $\omega^\omega$. We can find a $\rho<\nu\mu$ such that $\dot{x}$ is a $\Por_{\nu,\eta_\rho}$-name, so there is some $\zeta<\nu$ such that $\dot{x}=\dot{x}_{0,\zeta}^\rho$. Put $\xi=\eta_\rho+1+\varepsilon$ where $\varepsilon:=g^{-1}(0,\zeta,0)$ , so $\Por_{\Delta(\xi),\xi}$ forces that $\dot{x}\in \dot{N}_\xi$. Fix any $\beta\unrhd\xi$ in $R_0$. Then $\xi\leq\beta$, $\Delta(\xi)\leq\Delta(\beta)$ and $\Vdash_{\Por_{\Delta(\beta),\beta}}\dot{N}_\xi\subseteq \dot{N}_\beta$, so $\Vdash_{\Por_{\Delta(\beta),\beta}}\dot{x}\in\dot{N}_\beta$. Therefore, $\dot{\varphi}_\beta$ is forced to localize $\dot{x}$.

For each $\rho<\nu\mu$ denote by $\dot{e}_\rho$ the $\Por_{\Delta(\eta_\rho),\eta_\rho+1}$-name of the eventually different real over $V_{t(\rho)+1,\eta_\rho}$ added by $\Qnm_{{t(\rho)+1,\eta_\rho}}$. To show that  $\non(\Mcal)\geq\mu$ and $\cov(\Mcal)\leq\nu$, it is enough to prove that $\Por$ forces that $\dot{E}:=\{\dot{e}_\rho:\rho<\nu\mu\}$ is a strongly $\mu$-$\Ed$-dominating family. Consider the partial order on $\nu\mu$ defined by $\rho\unlhd'\varrho$ iff $\rho\leq\varrho$ and $t(\rho)\leq t(\varrho)$, which is actually ${<}\mu$-directed. To see this, let $A\subseteq\nu\mu$ of size of ${<}\mu$. Since $A$ is bounded with respect to $\leq$ (because $\cf(\nu\mu)=\mu$), in has an upper bound $\rho\in \nu\mu$. Define $\alpha:=\sup_{\eta\in A}\{t(\eta)+1\}$, which is ${<}\nu$ because $\nu$ is a regular cardinal. By the definition of $t$, there is some $\delta\in[\rho,\nu\mu)$ such that $\alpha=t(\delta)$, hence $\delta$ is an upper bound of $A$ with respect to $\unlhd'$.

Let $x\in V_{\nu,\pi}\cap\omega^\omega$. We can find $\alpha<\nu$ and $\rho<\nu\mu$ such that $x\in V_{\alpha,\eta_\rho}$. By the definition of $t$, there is some $\delta\in[\rho,\nu\mu)$ such that $t(\delta)=\alpha$, so $x\in V_{t(\delta),\eta_\delta}$. For any $\varrho \unrhd'\delta$, $\delta\leq\varrho$ and $t(\delta)\leq t(\varrho)$, so $x\in V_{t(\varrho)+1,\eta_\varrho}$, which implies $x\neq^*e_{\varrho}$.

To finish the proof we conclude that, by Theorem~\ref{matsizebd}, $\Por$ forces $\cov(\Mcal)=\dfrak(\Ed)\geq \nu$. In fact, if $c_\alpha$ denotes the Cohen real added by $\Qor_{\alpha+1,\alpha}$ for any $\alpha<\nu$, it is clearly $\Ed$-unbounded over $V_{\alpha,\alpha}=V_{\alpha,\alpha+1}$, so $\{c_\alpha:\alpha<\nu\}$ is a strongly $\nu$-$\Ed$-unbounded family.
\end{proof}

\begin{theorem} \label{Apl2}
Let $\theta_0\leq \theta_1\leq \mu\leq \nu$ be uncountable regular cardinals and let $\lambda$ be a cardinal such that $\nu\leq\lambda=\lambda^{{<}\theta_1}$. Then there is a ccc poset that forces $\mathrm{MA}_{{<}\theta_0}$, $\add(\Ncal)=\theta_0$, $\bfrak=\afrak=\theta_1$, $\cov(\Ncal)=\non(\Mcal)=\mu$, $\cov(\Mcal)=\non(\Ncal)=\nu$ and $\dfrak=\cof(\Mcal)=\cfrak=\lambda$ (see Figure~\ref{FigM} on page~\pageref{FigM}).
\end{theorem}

\begin{proof}
Fix a bijection $g=(g_0,g_1,g_2):\lambda\to2\times\nu\times\lambda$ and a function $t:\nu\mu\to\nu$ such that $t(\nu\delta+\alpha)=\alpha$ for each $\delta<\mu$ and $\alpha<\nu$. Denote $\eta_\rho:=\nu+\lambda\rho$ for each $\rho<\nu\mu$.

The desired poset is $\Hor_{\theta_1}\ast\Cor_\lambda\ast\Por$ where $\Por$ is constructed in $V_{0,0}=V^{\Hor_{\theta_1}\ast\Cor_\lambda}$ from a ${<}\theta_1$-uf-extendable matrix iteration $\mbf$.

Work in $V_{0,0}$. Put $I^\mbf:=\nu+1$, $\pi^\mbf:=\nu+\lambda\nu\mu$,

\begin{enumerate}[(I)]
    \item for any $\alpha<\nu$, $\Delta^\mbf(\alpha)=\alpha+1$ and $\Qnm^\mbf_{\Delta(\alpha),\alpha}=\omega^{<\omega}$,
\end{enumerate}

and define the matrix iteration in the intervals of the form $[\eta_\rho,\eta_{\rho+1})$ as follows. Assume that $\mbf{\upharpoonright}\eta_\rho$ has been defined. For $\alpha<\nu$ choose
\begin{itemize}
    \item[(0)] an enumeration $\{\Qnm_{0,\alpha,\zeta}^\rho:\zeta<\lambda\}$ of all the nice $\Por_{\alpha,\eta_\rho}$-names for all the posets which underlining set is a subset of $\theta_0$ of size ${<}\theta_0$ and $\Vdash_{\Por_{\nu,\lambda\rho}}$``$\Qnm_{0,\zeta}^\rho$ is ccc"; and
    \item[(1)] an enumeration $\{\Qnm_{1,\alpha,\zeta}^\rho:\zeta<\lambda\}$ of all the nice $\Por_{\alpha,\eta_\rho}$-names for all the $\sigma$-centered subposets of Hechler forcing of size ${<}\theta_1$.
\end{itemize}
For $\xi\in[\eta_\rho,\eta_{\rho+1})$,
\begin{enumerate}[(I)]
\setcounter{enumi}{1}
    \item if $\xi=\eta_\rho$ put $\Delta(\xi)=t(\rho)+1$ and $\Qnm_{\lambda\rho}^{\mbf}=\Bor^{V_{\Delta(\xi),\xi}}$;
    \item if  $\xi=\eta_\rho+1+\varepsilon$ for some $\rho<\nu\mu$ and $\varepsilon<\lambda$, put $\Delta(\xi)=g_1(\varepsilon)+1$ and $\Qnm_{\xi}^{\mbf}=\Qnm^{\rho}_{g(\varepsilon)}$.
\end{enumerate}
This settles the construction, which is clearly a ${<}\theta_1$-uf-extendable matrix iteration.
\end{proof}

\begin{remark}
   It is possible to additionally force $\mathrm{MA}_{{<}\theta_0}$ in Theorem~\ref{Apl1} by slightly modifying the construction of the matrix iteration. On the other hand, the matrix of Theorem~\ref{Apl2} could be modified to force the existence of a strongly-$\theta_0$-$\Lc^*$-dominating family and a strongly $\theta_1$-$\D$-dominating family.
\end{remark}

As a consequence of Theorem~\ref{Apl2}, we have a model where the cardinal invariants associated with any Yorioka ideal are pairwise different.

\begin{corollary}\label{Apl3}
Let $\theta\leq\mu\leq\nu$ be uncountable regular cardinals and let $\lambda\geq\nu$ be a cardinal such that $\lambda^{{<}\theta}=\lambda$. Then, as in Figure~\ref{FigI_f}, there is a ccc poset that forces $\add(\Ical_f)=\theta$, $\cov(\Ical_f)=\mu$, $\non(\Ical_f)=\nu$, and $\cof(\Ical_f)=\lambda$ for all increasing $f\in\omega^\omega$ (in the extension).
\end{corollary}
\begin{proof}
By application of Theorem~\ref{Apl2} to $\theta:=\theta_0=\theta_1$, there is a ccc poset that forces $\add(\Ncal)=\bfrak=\theta$, $\cov(\Ncal)=\non(\Mcal)=\mu$, $\cov(\Mcal)=\non(\Ncal)=\nu$ and $\dfrak=\cof(\Mcal)=\cfrak=\lambda$. This poset is as required by Theorem~\ref{YorioCichdiagram}.
\end{proof}

\begin{figure}
\begin{center}
  \includegraphics[scale=0.9]{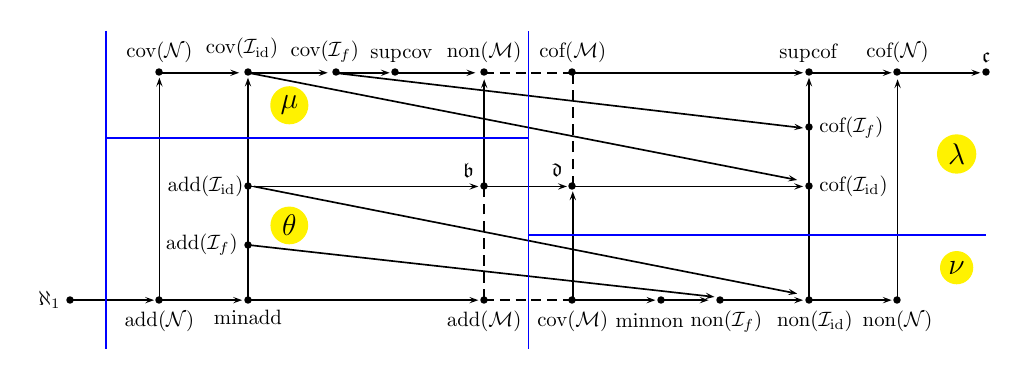}
  \caption{Separation of the cardinals associated with $\Iwf_f$ for any $f$.}
  \label{FigI_f}
\end{center}
\end{figure}

We finally show that Cicho\'n's diagram can consistently be separated into 10 values, assuming the consistency of three strongly compact cardinals. Though in~\cite{GKS} the same result is proved modulo four strongly compact cardinals and GCH, we avoid using GCH by tracking the exact necessary hypothesis about the cardinals.

\begin{theorem}\label{cichonmax3}
Assume:
\begin{enumerate}[(I)]
    \item $\kappa_9<\lambda_1<\kappa_8<\lambda_2<\kappa_7<\lambda_3\leq \lambda_4\leq \lambda_5\leq \lambda_6\leq \lambda_7\leq \lambda_8\leq \lambda_9$ are cardinal numbers,
    \item for $i\in [1,9]\menos\{6\}$, $\lambda_i$ is regular,
    \item $\lambda_6^{<\lambda_3}=\lambda_6$, and
    \item for $j\in\{7,8,9\}$, $\kappa_j$ is strongly compact and $\lambda_j^{\kappa_j}=\lambda_j$.
\end{enumerate}
Then there is a ccc poset that forces $\add(\Ncal)=\lambda_1$, $\cov(\Ncal)=\lambda_2$, $\bfrak=\lambda_3$, $\non(\Mcal)=\lambda_4$, $\cov(\Mcal)=\lambda_5$, $\dfrak=\lambda_6$, $\non(\Ncal)=\lambda_7$, $\cof(\Ncal)=\lambda_8$ and $\cfrak=\lambda_9$ (see Figure~\ref{FigCicmax}).
\end{theorem}

\begin{figure}
\begin{center}
  \includegraphics{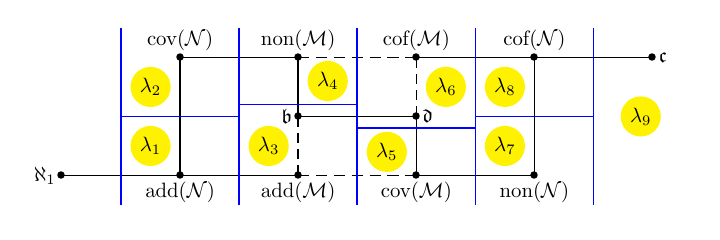}
  \caption{Cicho\'n's maximum}
  \label{FigCicmax}
\end{center}
\end{figure}

This result is justified by application of Boolean ultrapowers to the poset constructed in the proof of Theorem~\ref{Apl1} in the same way as in~\cite{KTT,GKS,KST}. We review this technique as follows. Let $\kappa$ be a strongly compact cardinal and $\lambda>\kappa$ regular such that\footnote{Without assuming GCH.} $\lambda^\kappa=\lambda$. Consider the Boolean completion $\Bor_{\kappa,\lambda}$ of the poset $\mathrm{Fn}_{{<}\kappa}(\lambda,\kappa)$ of partial functions from $\lambda$ to $\kappa$ with domain of size ${<}\kappa$ (ordered by $\supseteq$).

\begin{lemma}[{\cite{KTT,GKS}}]\label{BUP}
   There is a $\kappa$-complete ultrafilter $U$ on $\Bor_{\kappa,\lambda}$ such that its corresponding elementary embedding $j:V\to M$ satisfies:
   \begin{enumerate}[(a)]
   \item $M$ is closed under sequences of length ${<}\kappa$.
   \item $j$ has critical point $\kappa$, $\cf(j(\kappa))=\lambda$ and $\lambda\leq j(\kappa)<\lambda^+$.
   \item If $|A|<\kappa$ then $j[A]=j(A)$.
   \item If $\theta\geq\kappa$ and either $\theta\leq\lambda$ or $\theta^\kappa=\theta$, then $\max\{\lambda,\theta\}\leq j(\theta)<\max\{\lambda,\theta\}^+$.
   \item If $\theta>\kappa$ and $I$ is a ${<}\theta$-directed partial order then $j[I]$ is cofinal in $j(I)$.
   \item If $\cf(\alpha)\neq\kappa$ then $j[\alpha]$ is cofinal in $j(\alpha)$.
\end{enumerate}
\end{lemma}

As a consequence,

\begin{lemma}[{\cite{KTT,GKS}}, see also {\cite[Thm.~1.18]{GKMSproc}}]\label{elemembpres}
   Additionally to the above, assume that $\Rbf=\la X,Y,\sqsubset\ra$ is an analytic relational system (i.e., $X$, $Y$ and $\sqsubset$ are analytic in some Polish space), $\theta$ is an uncountable regular cardinal and that $\Por$ is a ccc poset. Then:
   \begin{enumerate}[(a)]
      \item $j(\Por)$ is ccc (in $V$, not just in $M$).
      \item If $\Por$ adds a strongly $\theta$-$\Rbf$-unbounded family of size $\theta$, then $j(\Por)$ adds a strongly $\cf(j(\theta))$-$\Rbf$-unbounded family of size $\cf(j(\theta))$.
      \item If $\Por$ adds a strongly $\theta$-$\Rbf$-dominating family with witnessing directed set $L$ in the ground model such that $|L|=\lambda'$, then
         \begin{enumerate}[(i)]
            \item whenever $\theta<\kappa$, $j(\Por)$ adds a strongly $\theta$-$\Rbf$ dominating family with witnessing directed set of size $|j(\lambda')|$;
            \item whenever $\kappa<\theta$, $j(\Por)$ adds a strongly $\theta$-$\Rbf$ dominating family with witnessing directed set of size $\lambda'$.
         \end{enumerate}
         In both cases, the witnessing directed set can be obtained in the ground model.
   \end{enumerate}
\end{lemma}
\begin{proof}
   We include the proof for completeness. Property (a) follows from Lemma~\ref{BUP}(a). To see property (b), let $\{\dot{c}(\alpha) : \alpha<\theta\}$ be a strongly $\theta$-$\Rbf$-unbounded family added by $\Por$. Since $\Por$ is ccc, $\exists \alpha<\theta\forall\beta\in[\alpha,\theta)(\Vdash_\Por \dot{c}(\beta)\not\sqsubset \dot{z})$ for any $\Por$-name $\dot{z}$ of a real in $Y$, thus
   \[M\models \exists \alpha<j(\theta)\forall\beta\in[\alpha,j(\theta))(\Vdash_{j(\Por)} j(\dot{c})(\beta)\not\sqsubset \dot{z}')\]
   for any $j(\Por)$-name $\dot{z}'$ of a real in $Y$ (note that every nice $j(\Por)$-name of a real is in $M$). Since $\Rbf$ is analytic, the same statement holds in $V$. Therefore, if $f:\cf(j(\theta))\to j(\theta)$ is an increasing cofinal function, then $\Por$ forces that $\{j(\dot{c})(f(\xi)): \xi<\cf(j(\theta))\}$ is a strongly $\cf(j(\theta))$-$\Rbf$-unbounded family.

   We finally show (c). Assume that $p\in\Por$ forces that $\{\dot{a}(l):l\in L\}$ is a strongly $\theta$-$\Rbf$-dominating family. Hence
   \[M\models ``j(p)\Vdash_{j(\Por)} \{j(\dot{a})(l):l\in j(L)\} \text{\ is a strongly $j(\theta)$-$\Rbf$-dominating family}''.\]
   If $\theta<\kappa$ then $j(\theta)=\theta$ and $j(L)$ is ${<}\theta$-directed (in $M$, but also in $V$) of size $|j(\lambda')|$; else, if $\kappa<\theta$, by Lemma~\ref{BUP}(e) we have that $j[L]$ is cofinal in $j(L)$, so $j(p)$ forces (in $V$) that $\{j(\dot{a})(j(l)):l\in L\}$ is a strongly $\theta$-$\Rbf$-dominating family.
\end{proof}

\begin{proof}[Proof of Theorem~\ref{cichonmax3}]
   Denote $\Rbf_0:=\Id$ $\Rbf_1:=\Lc^*$, $\Rbf_2:=\Cn$, $\Rbf_3:=\Dbf$, and $\Rbf_4:=\Ed$. Let $\Por_6$ be the poset constructed in Theorem~\ref{Apl1} applied to $\theta_i=\lambda_{i+1}$ for $i<3$, $\mu=\lambda_4$, $\nu=\lambda_5$ and $\lambda=\lambda_6$. Also let $\lambda_0:=\aleph_1$. Recall that $\Por_6$ adds
   \begin{itemize}
       \item[($\mathrm{U}_6$1)] a strongly $\kappa$-$\Rbf_i$-unbounded family of size $\kappa$ for $i<4$ and each regular $\kappa\in[\lambda_i,\lambda_6]$;
       \item[($\mathrm{U}_6$2)]\label{U62} a strongly $\lambda_i$-$\Rbf_4$-unbounded family of size $\lambda_i$ for $i\in\{4,5\}$;
       \item[($\mathrm{D}_6$1)] a strongly $\lambda_4$-$\Rbf_4$-dominating family with witnessing directed set of size $\lambda_5$ in the ground model; and
       \item[($\mathrm{D}_6$2)] a strongly $\lambda_i$-$\Rbf_i$-dominating family with witnessing directed set of size $\lambda_6$ in the ground model, for $1\leq i<4$.
   \end{itemize}
   Let $j_7:V\to M_7$ be the elementary embedding obtained from $\Bor_{\kappa_7,\lambda_7}$ as in the previous discussion, and let $\Por_7:=j_7(\Por_6)$. By Lemma~\ref{elemembpres}, $\Por_7$ is ccc and it adds
   \begin{itemize}
       \item[($\mathrm{U}_7$1)] a strongly $\kappa$-$\Rbf_i$-unbounded family of size $\kappa$ for $i<4$ and each regular $\kappa\in[\lambda_i,\lambda_6]\menos\{\kappa_7\}$;
       \item[($\mathrm{U}_7$2)] a strongly $\lambda_7$-$\Rbf_i$-unbounded family of size $\lambda_7$ for $i<3$
       \item[($\mathrm{U}_7$3)] a strongly $\lambda_i$-$\Rbf_4$-unbounded family of size $\lambda_i$ for $i\in\{4,5\}$;
       \item[($\mathrm{D}_7$1)] a strongly $\lambda_4$-$\Rbf_4$-dominating family with witnessing directed set of size $\lambda_5$;
       \item[($\mathrm{D}_7$2)] a strongly $\lambda_3$-$\Rbf_3$-dominating family with witnessing directed set of size $\lambda_6$; and
       \item[($\mathrm{D}_7$3)] a strongly $\lambda_i$-$\Rbf_i$-dominating family with witnessing directed set of size $\lambda_7$ for $1\leq i<3$.
   \end{itemize}
   This process is repeated a couple of times with $\kappa_8$ and $\kappa_9$. Let $j_8:V\to M_8$ be the elementary embedding obtained from $\Bor_{\kappa_8,\lambda_8}$ and set $\Por_8:=j_8(\Por_7)$. This poset is ccc and it adds
   \begin{itemize}
       \item[($\mathrm{U}_8$1)] a strongly $\kappa$-$\Rbf_i$-unbounded family of size $\kappa$ for $i<4$ and each regular $\kappa\in[\lambda_i,\lambda_6]\menos\{\kappa_7,\kappa_8\}$;
       \item[($\mathrm{U}_8$2)] a strongly $\lambda_7$-$\Rbf_i$-unbounded family of size $\lambda_7$ for $i<3$,
       \item[($\mathrm{U}_8$3)] a strongly $\lambda_8$-$\Rbf_i$-unbounded family of size $\lambda_8$ for $i<2$
       \item[($\mathrm{U}_8$4)] a strongly $\lambda_i$-$\Rbf_4$-unbounded family of size $\lambda_i$ for $i\in\{4,5\}$;
       \item[($\mathrm{D}_8$1)] a strongly $\lambda_4$-$\Rbf_4$-dominating family with witnessing directed set of size $\lambda_5$;
       \item[($\mathrm{D}_8$2)] a strongly $\lambda_3$-$\Rbf_3$-dominating family with witnessing directed set of size $\lambda_6$;
       \item[($\mathrm{D}_8$3)] a strongly $\lambda_2$-$\Rbf_2$-dominating family with witnessing directed set of size $\lambda_7$; and
       \item[($\mathrm{D}_8$4)] a strongly $\lambda_1$-$\Rbf_1$-dominating family with witnessing directed set of size $\lambda_8$.
   \end{itemize}
   Let $j_9:V\to M_9$ be the elementary embedding obtained from $\Bor_{\kappa_9,\lambda_9}$ and set $\Por_9:=j_9(\Por_8)$. This set is ccc and it satisfies the previous ($\mathrm{U}_8$1)--($\mathrm{U}_8$4) and ($\mathrm{D}_8$1)--($\mathrm{D}_8$4), with the exception that ($\mathrm{U}_8$1) does not hold for $\kappa=\kappa_9$. In addition, $\Por_9$ adds a strongly $\lambda_9$-$\Rbf_0$-unbounded family of size $\lambda_9$, so it forces $\lambda_9\leq\cfrak$ (see Example~\ref{ExmPrs}(5)). On the other hand, $|\Por_9|=|j_9(j_8(j_7(\lambda_6)))|=\lambda_9$, so $\Por_9$ forces $\cfrak\leq\lambda_9$.
   By the properties listed above, $\Por_9$ is the desired poset.
\end{proof}

\begin{remark}
   Due to the methods presented in~\cite{GKMSproc}, in Theorem~\ref{cichonmax3} we can weaken the conditions on $\lambda_9$: no need to assume that it is regular, and it satisfies
   $\lambda_9^{\aleph_0}=\lambda_9$.
\end{remark}

\section{Discussions}\label{SecDisc}

At the end of this section, we present some advancements related to this research that were made during the evaluation process of this paper.

In Theorem~\ref{Apl1} (Theorem~\ref{thm7val}) we separated one additional value in the right side of Cicho\'n's diagram with respect to the constellation proved in~\cite{GMS} (see (\ref{leftGMS}) in the introduction). We ask if we could do the same to the constellation from~\cite{KST}, concretely,

\begin{question}\label{Qanother7}
   Can it be forced, without using large cardinals, that
   \[\aleph_1<\add(\Nwf)<\bfrak<\cov(\Nwf)<\non(\Mwf)<\cov(\Mwf)<\non(\Nwf)=\dfrak=\cfrak\text{?}\]
\end{question}
If this is possible, the large cardinal hypothesis from the main result in~\cite{KST} can be reduced to three strongly compact cardinals.

The matrix iteration technique of this text seem not to be enough to deal with this problem since, to give desired values to $\cov(\Nwf)$ and $\non(\Mwf)$ without increasing $\bfrak$ too much, we need to deal with restrictions of random forcing and $\Eor$ simultaneously, so they cannot be included in the same way in the matrix construction (a bit more in detail, only one could be the restriction to $V_{\Delta(\xi),\xi}$, but the other must be other type of restriction). On the other hand, similar to~\cite{KST}, dealing with ultrafilters may not be enough, so the matrix construction may include finitely additive measures instead.

The reader may have noticed that we did not force a value of $\afrak$ in Theorem~\ref{cichonmax3} after using Boolean ultrapowers. The reason is that the Boolean ultrapowers from $\Bor_{\kappa,\lambda}$ applied to a ccc poset $\Por$ destroys all the mad families of size $\geq\kappa$ added by $\Por$ in the same way as the ultrapower from a measurable cardinal destroys them (see~\cite{Sh:700,brendle02}). This leads us to ask whether a value of $\afrak$ can be forced in Theorem~\ref{cichonmax3}, or even in the consistency results from~\cite{GKS,KST}.

By a slight modification, the poset constructed in Theorem~\ref{Apl1} can force $\mathrm{MA}_{<\lambda_1}$ (with $\lambda_1=\theta_0$), and it is not hard to see that $\Por_8$ from the proof of Theorem~\ref{cichonmax3} also forces this. Though we can  guarantee that $\Por_9$ forces $\mathrm{MA}_{{<}\kappa_9}$, it is unclear whether it forces $\mathrm{MA}_{<\lambda_1}$.

\subsection*{Updates} As mentioned in the introduction, \cite{GKMS} showed that Cicho\'n's diagram can consistently be separated into 10 values without assuming large cardinals, concretely, for the instances (\ref{maxGKS}) and (\ref{maxKST}) (see Section~\ref{SecIntro}, in particular Question~\ref{Qanother7} is solved in the positive). The dynamic of the proof is similar to the original~\cite{GKS}: start with a ccc poset $\Por^0$ that separates the left hand side of Cicho\'n's diagram, e.g.\ (\ref{leftKST}) and (\ref{leftGMS}), but instead of taking Boolean ultrapowers, intersect $\Por^0$ with $\sigma$-closed elementary submodels of $H_\chi$ (for some large enough regular $\chi$) so that the resulting poset forces Cicho\'n's diagram separated into 10 values.

This new method still relies on a forcing that separates the left side of the diagram. Although the poset from~\cite{GMS} does this job, the new method is incompatible with conditions (P1)--(P3) (see Section~\ref{SecIntro}), so a modification of this forcing as in~\cite{GKS} is necessary to get a poset compatible with the new method. The same happens with the Boolean ultrapower method.  On the other hand, the poset we presented in Theorem~\ref{Apl1} is already compatible with the new method, and less difficult to construct in comparison with the forcing from~\cite{GMS,GKS}.

In relation with the previous discussion about MA, \cite{GKMSmore} shows how to separate other classical cardinal characteristics of the continuum (without using large cardinals), in addition to those in Cicho\'n's diagram. In particular, $\mathfrak m$ (the smallest cardinal where MA fails) can be forced to be any chosen regular value between $\aleph_1$ and (the intended) $\add(\Nwf)$. With respect to the Boolean ultrapower method, in~\cite{GKMSproc} the forcing from Theorem~\ref{cichonmax3} is modified to force, in addition, that $\mathfrak m$ can be any previously chosen regular value between $\aleph_1$ and $\kappa_9$.

{\small
\bibliography{bibli}
\bibliographystyle{alpha}
}

\end{document}